\documentclass{amsart}

\usepackage{amssymb,amscd,a4wide,graphicx,tikz,hyperref,amsbsy,mathrsfs,enumerate,mathtools,stmaryrd}
\usepackage{algorithm,algorithmic}
\usepackage{hyperref}
\usepackage[all]{xy}
\usepackage{color}
\usepackage{todonotes}

\title[Self-affine tiles and Rauzy Fractals]{Neighbors of self-affine tiles and Rauzy Fractals}
\author{Beno\^it Loridant}
\author{J\"org M. Thuswaldner}
\author{Shu-Qin Zhang$^\ast$}
\thanks{$^\ast$Corresponding author}
\thanks{The second author is supported by the ANR-FWF grant I 6750. The third author is supported by NSFS No. 12101566}
\address[B.L. \& J.M.T.]{Chair of Mathematics and Statistics, University of Leoben, Franz-Josef-Strasse 18, A-8700 Leoben, Austria}
\email{benoit.loridant@unileoben.ac.at}
\email{joerg.thuswaldner@unileoben.ac.at}
\address[S.-Q.Z.]{School of Mathematics and Statistics, Zhengzhou University, 100 Science Avenue, Zhengzhou, Henan 45001, People’s Republic of China}
\email{sqzhang@zzu.edu.cn}
\dedicatory{Dedicated to Professor Shigeki Akiyama on the occasion of his 65\textsuperscript{th} birthday.}

\newtheorem{lemma}{Lemma}[section]
\newtheorem{theorem}[lemma]{Theorem}

\newtheorem{proposition}[lemma]{Proposition}

\theoremstyle{definition}
\newtheorem{definition}[lemma]{Definition}

\theoremstyle{remark}

\newtheorem{example}[lemma]{Example}

\numberwithin{equation}{section}

\newcommand{\assign}{\leftarrow}

\newcommand{\N}{\mathbb{N}}
\newcommand{\Z}{\mathbb{Z}}

\newcommand{\R}{\mathbb{R}}

\newcommand{\cA}{\mathcal{A}}
\newcommand{\cB}{\mathcal{B}}

\newcommand{\cP}{\mathcal{P}}
\newcommand{\cR}{\mathcal{R}}
\newcommand{\cS}{\mathcal{S}}
\newcommand{\cT}{\mathcal{T}}

\newcommand{\ba}{\mathbf{a}}
\newcommand{\bb}{\mathbf{b}}
\newcommand{\bd}{\mathbf{d}}
\newcommand{\be}{\mathbf{e}}
\newcommand{\bm}{\mathbf{m}}

\newcommand{\bu}{\mathbf{u}}
\newcommand{\bv}{\mathbf{v}}
\newcommand{\bw}{\mathbf{w}}
\newcommand{\bx}{\mathbf{x}}
\newcommand{\by}{\mathbf{y}}
\newcommand{\bz}{\mathbf{z}}

\newcommand{\xdownarrow}[1]{%
  {\left\downarrow\vbox to #1{}\right.\kern-\nulldelimiterspace}
}

\date{\today}

\begin{document} 

\begin{abstract}
Although the theory of self-affine tiles and the theory of Rauzy fractals are quite different from each other, they have some common features. Both, self-affine tiles and Rauzy fractals have tiling properties and these tiling properties can be checked and described by certain graphs, so-called {\it contact graphs} and {\it neighbor graphs}. The contact graph is often quite easy to construct, but only the neighbor graph contains full information on the overlaps of the tiles in the presumed tiling. In the present paper we establish an algorithm that allows to construct the neighbor graph starting from the contact graph. Such an algorithm is already known in the case of self-affine tiles. In the present paper we give a simplified proof of this algorithm that can be extended to the case of Rauzy fractals. Our algorithms are more efficient than na\"ive algorithms for the construction of the neighbor graph.
\end{abstract}

\maketitle

\setcounter{tocdepth}{1}
\tableofcontents

\section*{Introduction} 
In the present paper we study self-affine tiles and Rauzy fractals. For both of these objects, there exists a vast literature. The foundations of the theory of self-affine tiles were laid by Kenyon~\cite{Kenyon:92}, Gr\"ochenig and Haas~\cite{GroechenigHaas:94}, as well as Lagarias and Wang~\cite{LW:96,LW:96a,LW:97} in the 1990s and since then these objects have been studied extensively. Important topics include periodic tiling properties~\cite{LW:97}, the characterization of digit sets~\cite{AL:19,Lai-Lau:17,Lai-Lau-Rao:17}, topological questions~\cite{CT:16,DLN:22,TZ:19},  relations to Fuglede's conjecture~\cite{LQT:23}, as well as relations to wavelets~\cite{FG:15,GM:92}. Rauzy fractals go back to Rauzy~\cite{Rauzy:82} and were studied in a more general framework in \cite{Arnoux-Ito:01}.  They play a role in many branches of mathematics like for instance in the theory of aperiodic order \cite{BG:13,BG:20}, in the theory of symbolic dynamical systems~\cite{ABBLS:15,CANTBST,Ito-Rao:06} and in number theory~\cite{Akiyama:02,BS:05,BST:19}. Also, Rauzy fractals often admit different kinds of tilings~\cite{Ito-Rao:06}, such as a periodic tiling and a self-replicating tiling.

Although the theory for self-affine tiles and Rauzy fractals have been developped in different directions, they have common features. As already mentioned, both, self-affine tiles and Rauzy fractals, often admit tilings. And even more, in both cases, so-called {\em contact graphs} and {\em neighbor graphs} can be used in order to check tiling properties. Contact graphs are defined in \cite{GroechenigHaas:94} for the case of self-affine tiles and in \cite{ST:09} for the case of Rauzy fractals. In \cite{LW:97} a Fourier transform of the contact graph was used in order to establish a general tiling result for self-affine tiles. The analogue of such a tiling result remains open for Rauzy fractals and is related to the famous {\em Pisot conjecture} \cite{ABBLS:15}, although there exist many partial results (see for instance~\cite{Arnoux-Ito:01,Barge:16,Barge:18,BD:02}). We therefore think that it is interesting to look at contact and neighbor graphs for self-affine tiles and Rauzy fractals at the same time, and we will do this in the present paper. Although we do not think that the methods leading to the general tiling result of \cite{LW:97} can be carried over to the case of Rauzy fractals, we hope that analogies between the tiling theory of self-affine tiles and Rauzy fractals can help to shed some light on the Pisot conjecture.

For self-affine tiles as well as for Rauzy fractals it is fairly easy to show that they give rise to multi-tilings. To show that these multi-tilings are actually tilings turns out to be the hard part. To this end, one has to show that the tiles do not overlap in sets of positive measure. Since these overlaps are captured by the contact graph and, even more precisely, by the neighbor graph associated to the self-affine tile or Rauzy fractal, these graphs naturally play a role when it comes to tiling properties (see \cite{GroechenigHaas:94,ST:09}). The contact graph is often considerably smaller and easier to construct than the neighbor graph. However, only the neighbor graph contains full information on the overlaps. For this reason we developed an algorithm that allows to construct the neighbor graph from the contact graph in the case of self-affine tiles \cite{ST:03}. This algorithm is useful in order to characterize topological properties of self-affine tiles \cite{AT:05,TZ:20} because, contrary to na\"ive algorithms, it can be used to obtain the neighbor graph for whole classes of self-affine tiles. In the present paper we provide a simplified proof of this algorithm that can be extended to the case of Rauzy fractals. We think that our new algorithm for Rauzy fractals can be used to characterize ``neighbors'' for whole classes of Rauzy fractals and forms the basis for getting topological information on Rauzy fractals in dimension $3$ and higher. Even for single examples like the natural generalization of Rauzy's classical ``tribonacci substitution'' studied in \cite{DM:11} it would be interesting to prove topological results.

We mention that, according to \cite{SirventWang02}, Rauzy fractals are solutions of graph directed iterated function systems in the sense of Mauldin and Williams~\cite{Mauldin-Williams:88}. Thus, one could ask if our results can be extended to more general tiles that are solutions of graph directed iterated function systems, like for instance the ones studied in \cite{GHR:99,LW:03}. However,  our methods rely on an additional structure of Rauzy fractals that is not present in the case of more general tiles.

\subsection*{Outline} In Section~\ref{sec:SelfAffine} we provide the necessary definitions around self-affine tiles. After that we state the algorithm that constructs the neighbor graph of a self-affine tile starting from its contact graph. We then give a simplified proof that the algorithm outputs the neighbor graph after finitely many steps. Section~\ref{sec:RauzyBasic} contains all relevant definitions and basic results on Rauzy fractals and their tiling properties. In Section~\ref{sec:RGRAPHS} we define the neighbor and contact graph for Rauzy fractals. To be consistent with the literature, in the setting of Rauzy fractals, the neighbor graph is called {\em self-replicating boundary graph}. After recalling the definitions given already in \cite[Section~5]{ST:09}, we provide simplified versions of these graphs that are better suited for our purposes. Section~\ref{sec:RALGOR} contains the statement and proof of the new algorithm that allows to construct the self-replicating boundary graph for a Rauzy fractal from its contact graph. Although its proof is inspired by our proof for the self-affine case from Section~\ref{sec:SelfAffine}, the details are more involved. The paper ends with Section~\ref{sec:Examples}, where our new algorithm is applied to two examples.

\subsection*{Notation} Throughout the paper, $\Vert\cdot\Vert$ denotes some norm on $\R^d$. The induced operator norm is also denoted by $\Vert\cdot\Vert$. Let $G$ be a directed graph. Then the graph $\mathop{Red}(G)$ is the largest subgraph of $G$ that has no node without an outgoing edge.  Let $L$ be a linear subspace of $\R^d$ equipped with the Lebesgue measure. We say that a collection $\mathcal{K}$ of compact sets with nonempty interior forms a {\em multi-tiling} of $L$ if there is $m\in \N$ such that almost every point of $L$ is contained in exactly $m$ elements of $\mathcal{K}$. If $m=1$, we call $\mathcal{K}$ a {\em tiling} of $L$.  

\section{The neighbor finding algorithm for self-affine tiles revisited} \label{sec:SelfAffine}

\subsection{Self-affine tiles with standard digit set}
Let $M$ be a $d\times d$ expanding integer matrix, {\it i.e.}, a $d\times d$ integer matrix each of whose eigenvalues is greater than $1$ in modulus. Suppose that $\mathcal{D}$ is a complete set of coset representatives of $\Z^d/M\Z^d$. Then by an application of the contraction mapping theorem~\cite{H:81} we know that there exists a unique nonempty compact subset $\cT=\cT(M,\mathcal{D})$ of $\R^d$ such that
\begin{equation}\label{setequ:affinetile}
\cT=\bigcup_{\bd\in\mathcal{D}}M^{-1}(\cT+\bd).
\end{equation}
We call $\cT$ a {\em self-affine tile with standard digit set $\mathcal{D}$}.  Bandt~\cite{Bandt91} showed that the Lebesgue measure of a self-affine tile with standard digit set is positive. These tiles have been studied extensively; Figure~\ref{Ex:self-affine tiles} provides two examples. For a survey on earlier results on the topic we refer to \cite{Wang99}. A more recent survey is provided in \cite{Lai-Lau:17}.   
\begin{figure}[h]
\includegraphics[width=2.2 cm]{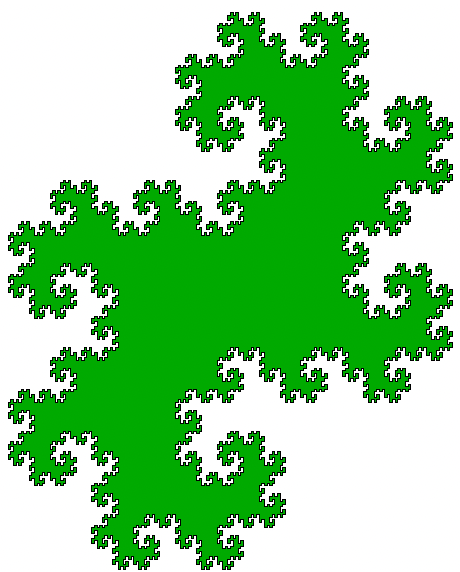} \quad\quad\quad\quad  \includegraphics[width= 3.2 cm]{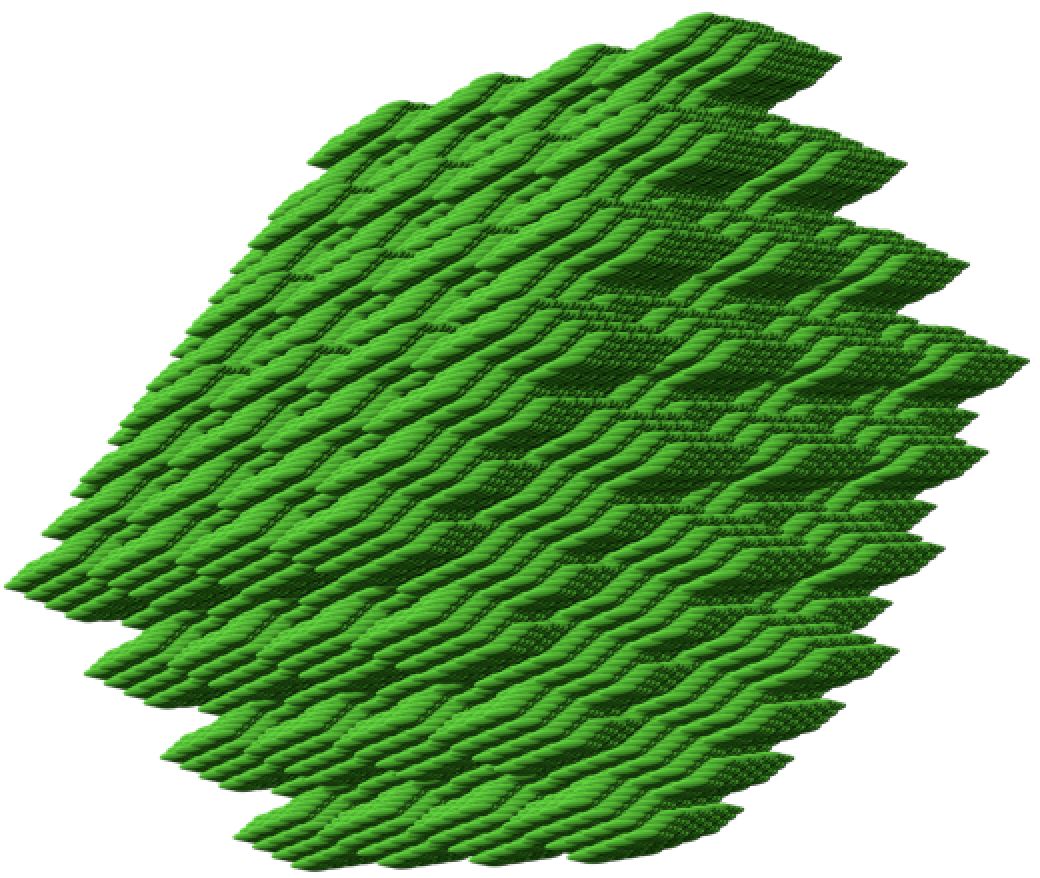}
\caption{Knuth's {\em twindragon}, a famous self-affine tile (see {\it e.g.}~\cite[p.~608]{Knuth:98}), and a three-dimensional self-affine tile.}\label{Ex:self-affine tiles}
\end{figure}

By \cite[Corollary~1.1b]{LW:96}, for each self-affine tile $\cT=\cT(M,\mathcal{D})$ with standard digit set the collection $\{\cT+\bz \colon \bz\in \Z^d\}$ forms a multi-tiling. Often, this collection even forms a tiling. A general characterization of this tiling property is contained in \cite{LW:97}.

\subsection{The neighbor graph and the contact graph}
Let $\cT=\cT(M,\mathcal{D})$ be a self-affine tile with standard digit set. For $\bm\in \Z^d$ let $\cS[\bm]: =\cT\cap (\cT+\bm)$ be the ``overlap'' of $\cT$ with its translate by $\bm$. The \emph{neighbor set} $S$ of $\cT$ is the set of all nonzero $\Z^d$-translates $\bm$ for which the overlap $\cS[\bm]$ is nonempty, {\em i.e.}, 
$$
S=\big\{\bm\in \Z^d\setminus\{\mathbf{0}\}\;: \; \cS[\bm]\neq \emptyset\big\}.
$$
The set $S$ is  finite by compactness of $\cT$. The terminology ``neighbor set'' is geometrically justified because if $\{\cT+\bz \colon \bz\in \Z^d\}$ forms a tiling, then $S$ consists of all translates $\bm$ for which  $\cT +\bm$ ``touches'' $\cT$. By the definition of $\cT$ via the set equation \eqref{setequ:affinetile} it turns out that  
\begin{equation*}
\begin{split}
\label{GIFS:boundary}
\cS[\bm] 
=\cT\cap (\cT+\bm)
 &
=M^{-1}(\cT+\mathcal{D})\cap M^{-1}(\cT+\mathcal{D}+M\bm)
\\  &
=M^{-1}\bigcup_{\bd,\bd'\in \mathcal{D}}(\cS[M\bm+\bd'-\bd]+\bd)     
\end{split}
\end{equation*}
holds for each $\bm\in S$. Thus $(\cS[\bm])_{\bm\in S}$ is the unique vector of nonempty compact sets that is given by the graph-directed iterated function system (in the sense of~\cite{Mauldin-Williams:88})
\begin{equation}
\label{GIFS:boundary}
\cS[\bm]
=M^{-1}\bigcup_{\begin{subarray}{c}
\bd,\bd'\in\mathcal{D},\, \bm'\in S \\ \bm'=M\bm+\bd'-\bd
\end{subarray}}(\cS[\bm']+\bd)      \qquad(\bm\in S).
\end{equation}

If $\{\cT+\bz \colon \bz\in \Z^d\}$ forms a tiling, the boundary of $\cT$ can be written as a union of these overlaps, in particular,
$$
\partial \cT=\bigcup_{\bm\in \cS} \cS[\bm].
$$

For each $A\subset \Z^d$ we define a graph $\Gamma_A$ in the following way. The vertices of $\Gamma_A$ are the elements of $A$ and there exists a labeled edge $ \bm\xrightarrow{\bd|\bd'} \bm'$ from $\bm\in A$ to $\bm'\in A$ with label $\bd|\bd' \in \mathcal{D}\times\mathcal{D}$ if and only if 
$
\bm'=M\bm+\bd'-\bd. 
$

Let $S\subset \Z^d$ be the neighbor set defined above. Then we call the graph $\Gamma_S$ the \emph{neighbor graph} of $\cT=\cT(M,\mathcal{D})$ ({\it cf.}~\cite{ST:03}). Hence, the set equation \eqref{GIFS:boundary} can be written as 
\begin{equation*}
\cS[\bm]
=M^{-1}\bigcup_{
\begin{subarray}{c}
\bd,\bd'\in\mathcal{D},\, \bm'\in S \\ \bm\xrightarrow{\bd|\bd'} \bm' \in \Gamma_S
\end{subarray}}(\cS[\bm']+\bd)      \qquad(\bm\in S).
\end{equation*}
 
The following characterization of the neighbor graph is an immediate consequence of \cite[Lemma~4.2 and its proof]{ST:03}.

\begin{lemma}\label{lem:tileendsinloop}
The graph $\Gamma_{S\cup\{\mathbf{0}\}}$ is the largest subgraph of $\Gamma_{\mathbb{Z}^d}$ for which each node belongs to a walk that ends in a loop. 
\end{lemma}
 
To study the neighbor graph, we introduce another graph which turns out to be a subgraph of $\Gamma_{S\cup\{\mathbf{0}\}}$. Let $\{\be_1,\be_2,\dots,\be_d\}$ be a basis of the lattice $\Z^d$ and set $R_0=\{0,\pm \be_1,\dots,\pm \be_d\}$. Then we inductively define a nested sequence $(R_k)_{k\geq 0}$ of $\Z^d$  by
\begin{equation}\label{eq:contact}
R_k:=\big\{\by\in \Z^d\;:\; (M\by+\mathcal{D})\cap(\by'+\mathcal{D})\neq\emptyset\text{ for } \by' \in R_{k-1}\big\}\cup R_{k-1}.
\end{equation}
The sequence $(R_k)_{k\geq 0}$ will stabilize after finitely many steps which means that $R_{k-1}=R_k$ holds for $k$ large enough (see~\cite[Section 4]{GroechenigHaas:94}).  Let $R' =\bigcup_{k\geq 1} R_k$ and define the {\em contact set} $R\subset R'$ by $\Gamma_{R}=\mathop{Red}(\Gamma_{R'})$. The graph $\Gamma_R$ is called the \emph{contact graph}.

Let $\cT_0=[0,1]^d$ be the $d$-dimensional unit square. Set
$$
\cT_n=M^{-1}(\cT_{n-1}+\mathcal{D}) \qquad(n\ge 1).
$$
Then  $\cT_n$ can be written as 
\begin{align*}
\cT_n&=M^{-n}(\cT_0+\mathcal{D}+M\mathcal{D}+\dots+M^{n-1}\mathcal{D}).
\end{align*}
It is clear that $\cT= \lim_{n\to \infty}\cT_n$, where the limit is taken w.r.t.\ the Hausdorff metric; see~\cite{GroechenigHaas:94}. We mention that $\{\cT_n+\bz\colon\bz\in\Z^d\}$ forms a tiling of $\R^d$ for each $n\geq 0$ even if $\{\cT+\bz\colon\bz\in\Z^d\}$ only forms a multi-tiling (this can be proved by induction on $n$ because $\mathcal{D}$ is a complete set of coset representatives of $\Z^d/M\Z^d$). 
The elements of $R$ are the ``neighbors'' of $\cT_n$ in this tiling in the sense that $\cT_n\cap (\cT_n+\bm)\not=\emptyset$ if and only if $\bm\in R$ (provided that $n$ is large enough; see \cite{ST:03} for details).

\subsection{The neighbor finding algorithm}
We will now describe an algorithm that will allow us to compute the neighbor graph starting from the contact graph. The algorithm is the same as the one presented in \cite{ST:03}, however, we present it in a somewhat simpler way. Indeed, instead of the graph product ``$\otimes$'' defined in \cite{ST:03}, we use the following concept of {\em $R$-corona}.

\begin{definition}[$R$-corona]
Let $\cT=\cT(M,\mathcal{D})$ be a self-affine tile with standard digit set and let $R$ be its contact set.
Let $A\subset \Z^d$ be given. The {\em $R$-corona} of $\Gamma_A$ is the graph $\Gamma_{A+R}$, where $A+R$ is the Minkowski sum. 
\end{definition}

Note that, in the notation of \cite{ST:03} we have $\Gamma_{A+R}=\Gamma_A\otimes \Gamma_R$.

\begin{algorithm}\
  \caption{(Construction of the neighbor graph of a self-affine tile with standard digit set $\cT$.)}
   \label{alg:Tileneighbor}
  \begin{algorithmic}
    \REQUIRE Contact graph $\Gamma_R$
    \ENSURE Neighbor graph $\Gamma_S$
    \STATE $q\assign 1$ 
    \STATE $\Gamma_{R_1} \assign \Gamma_R$
     \REPEAT 
      \STATE $q\assign q+1$
      \STATE $\Gamma_{R_q} \assign\mathop{Red}(\Gamma_{R_{q-1}+R})$
    \UNTIL{$\Gamma_{R_q} = \Gamma_{R_{q-1}}$}
    
    $\Gamma_S \assign \Gamma_{{R_q} \setminus \{\mathbf{0}\}}$
  \end{algorithmic}
 \end{algorithm}

\subsection{A simplified proof for the algorithm}
We now provide a simplified proof of Algorithm~\ref{alg:Tileneighbor}. This proof will carry over to the case of self-replicating tilings induced by Rauzy fractals and, therefore, yields a faster algorithm for the computation of the neighbors of a Rauzy fractal than the algorithm proposed in~\cite{ST:09}. We use the notation of \cite{ST:03}. The tile used in Figures~\ref{fig:FigContact} and~\ref{fig:FigSubdivision} is given by $\cT=\cT(M,\mathcal{D})$ with
\[
M=\begin{pmatrix}
2&-1 \\ 1&2
\end{pmatrix} \quad
\text{and} \quad
\mathcal{D}=\left\{
\begin{pmatrix}
0\\0
\end{pmatrix},\begin{pmatrix}
1\\0
\end{pmatrix},\begin{pmatrix}
2\\0
\end{pmatrix},\begin{pmatrix}
3\\0
\end{pmatrix},\begin{pmatrix}
4\\0
\end{pmatrix}
\right\}.
\]

By \cite[Lemma~4.4]{ST:03}, $R$ contains a basis of the lattice $\mathbb{Z}^d$. Thus each $\bm\in \Z^d$ can be written as
\begin{equation}\label{eq:Rsubdivision}
\bm= \sum_{j=1}^q \bm^{(j)} \qquad (\bm^{(1)},\ldots, \bm^{(q)} \in R).
\end{equation}
If $q$ is chosen to be minimal, we call $q$ the {\em contact degree} of $\bm$ and write $\mathop{cdeg}(\bm)=q$. Since $S$ is a finite set, there is a constant $p$ such that each $\bm\in S$ has $\mathop{cdeg}(\bm)\le p$. 
Geometrically this means that the tiles $\cT + \sum_{j=1}^{k-1} \bm^{(j)}$ and $\cT + \sum_{j=1}^k \bm^{(j)}$, $k\in\{1,\ldots, p\}$, are contact neighbors and we can reach $\cT+\bm$, $\bm\in S$, from $\cT$ by $q\le p$ ``contact neighbor jumps'' (see the illustration in Figure~\ref{fig:FigContact}).
\begin{figure}[h]
\begin{tikzpicture}[]
  \pgftext{\includegraphics[width=0.45\textwidth]{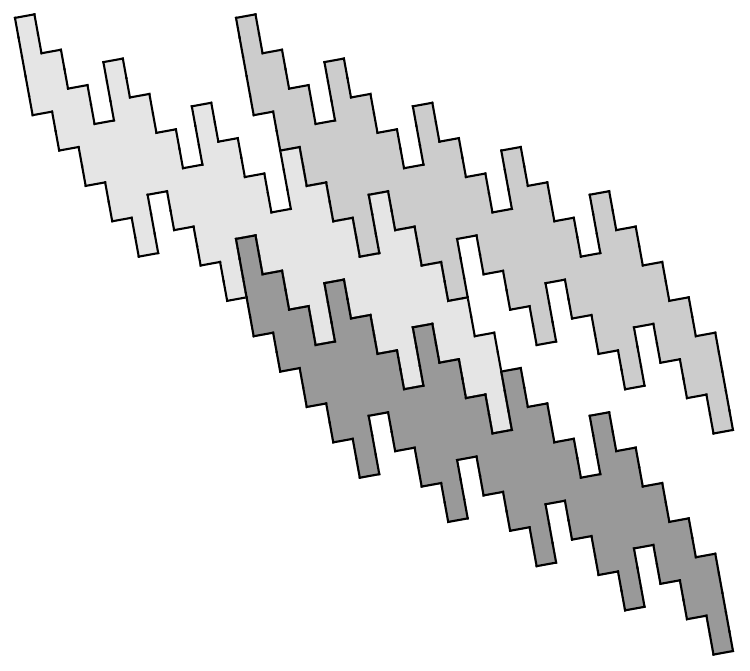}} at (0pt,0pt);
  \node at (0.25,-1.8) {$\mathcal{T}_3$};
  \node at (-2.25,0.3) {$\mathcal{T}_3+\bm^{(1)}$};
  \node at (1.5,2.5) {$\mathcal{T}_3+\bm^{(1)}+\bm^{(2)}$};
\end{tikzpicture}
\begin{tikzpicture}[]
  \pgftext{\includegraphics[width=0.45\textwidth]{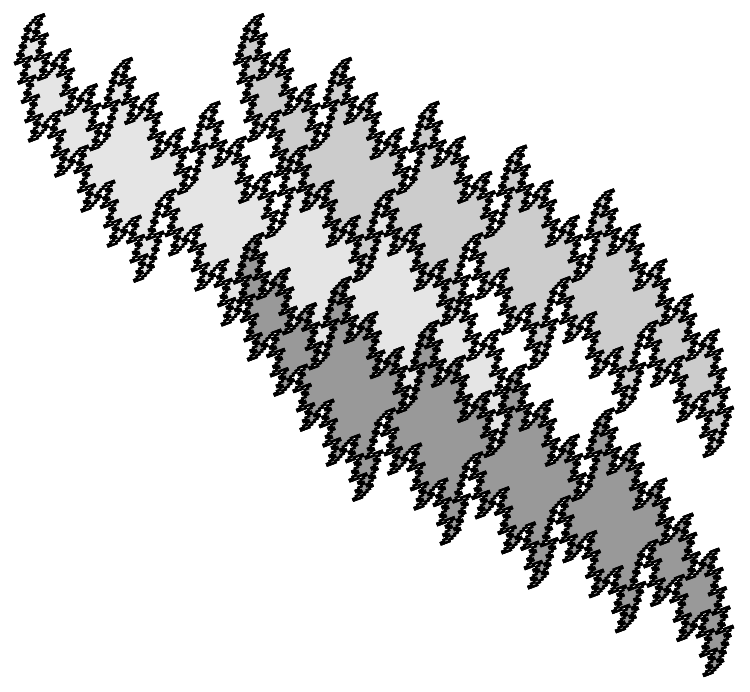}} at (0pt,0pt);
  \node at (0.25,-1.8) {$\mathcal{T}$};
  \node at (-2.35,0.35) {$\mathcal{T}+\bm^{(1)}$};
  \node at (1.5,2.5) {$\mathcal{T}+\bm^{(1)}+\bm^{(2)}$};
\end{tikzpicture}
\caption{In the 3rd approximation one can jump from $\cT_3$ via $\cT_3+\bm^{(1)}$ to $\cT_3+\bm^{(1)}+\bm^{(2)}$; the tile $\cT_3+\bm^{(1)}+\bm^{(2)}$ is not a contact neighbor of $\cT_3$ (left panel). However, $\cT+\bm^{(1)}+\bm^{(2)}$ is a neighbor of $\cT$. In the right panel we can see that $\cT+\bm^{(1)}+\bm^{(2)}$ is a neighbor of $\cT$, {\it i.e.}, $\cT \cap (\cT+\bm^{(1)}+\bm^{(2)})\not=\emptyset$.\label{fig:FigContact}}
\end{figure}
Note that the representation \eqref{eq:Rsubdivision} is not unique. 

Let
\begin{equation}\label{eq:walkS}
w:\; \bm_0 \xrightarrow{\bd_1| \bd'_1} \bm_1 \xrightarrow{\bd_2| \bd'_2} \bm_2 \xrightarrow{\bd_3| \bd'_3} \cdots
\end{equation}
be a walk in $\Gamma_{\Z^d}$. The maximum $\mathop{cdeg}(w) = \max\{\mathop{cdeg}(\bm_k)\colon k\in \N\}$, which may be infinite, is called the {\em contact degree} of $w$. Note that, for a walk $w$ in $\Gamma_S$, we always have $\mathop{cdeg}(w) \le p$ and that $\mathop{cdeg}(w) =1$ means that $w$ is a walk in $\Gamma_R$.  

 It turns out that the contact degree of the nodes of $w$ is monotone in the following sense. 

\begin{lemma}\label{lem:mono}
Let a walk $w$ in $\Gamma_{\Z^d}$ be given as in \eqref{eq:walkS}. Then $\mathop{cdeg}(\bm_k) \le \mathop{cdeg}(\bm_{k+1})$ holds for each $k\in\N$.
\end{lemma}

\begin{proof}
Let $q= \mathop{cdeg}(\bm_{k+1})$. Then there exist $\bm_{k+1}^{(1)},\ldots, \bm_{k+1}^{(q)} \in R$ with $\bm_{k+1}=\bm_{k+1}^{(1)}+\cdots+ \bm_{k+1}^{(q)}$. Let $\bd_{k+1}^{(0)}=\bd_{k+1}$. Because $\mathcal{D}$ is a complete set of coset representatives of $\Z^d/M\Z^d$, for $j\in\{1,\ldots, q\}$ we may recursively define $(\bm_k^{(j)},\bd_{k+1}^{(j)})\in \Z^d\times\mathcal{D}$ satisfying 
\begin{equation}\label{eq:contactdecomp}
\bm_{k+1}^{(j)}+\bd_{k+1}^{(j-1)}=M\bm_k^{(j)} +\bd_{k+1}^{(j)} . 
\end{equation}
Thus, by the definition of the contact graph (see, in particular, \eqref{eq:contact}) we have $\bm_k^{(j)} \in R$ and, summing \eqref{eq:contactdecomp} over all $j\in\{1,\ldots,q\}$ we get
\begin{equation}\label{eq:contactdecomp2}
\bm_{k+1}+ \bd_{k+1}=M(\bm_k^{(1)}+\cdots+\bm_{k}^{(q)}) +\bd_{k+1}^{(q)}. 
\end{equation}
Because the existence of the edge $\bm_k \xrightarrow{\bd_{k+1}| \bd'_{k+1}} \bm_{k+1}$ implies that 
$\bm_{k+1}+ \bd_{k+1}=M\bm_k +\bd_{k+1}'$, by \eqref{eq:contactdecomp2} and the fact that $\mathcal{D}$ is a complete set of coset representatives of $\Z^d/M\Z^d$ we gain that $\bm_{k}=\bm_{k}^{(1)}+\cdots+ \bm_{k}^{(q)}$ (and $\bd_{k+1}^{(q)}=\bd_{k+1}'$). Because $\bm_k^{(j)} \in R$ and it is possible that $\bm_k^{(j)}=\mathbf{0}$ for some $j\in\{1,\ldots,q\}$, this implies that $\mathop{cdeg}(\bm_{k})\le q$, and the result follows.
\end{proof}
 
 The following proposition contains the key step of our proof.
 
\begin{proposition}\label{lem:2}
Let
\begin{equation}\label{eq:walkSrep}
w:\;\bm_0 \xrightarrow{\bd_1| \bd'_1}\bm_1 \xrightarrow{\bd_2| \bd'_2}\bm_2 \xrightarrow{\bd_3| \bd'_3} \cdots
\end{equation}
be a walk in $\Gamma_{S\cup \{\mathbf{0}\}}$ with $\mathop{cdeg}(w)>1$. Then there exist $\bd_i'' \in \mathcal{D}$ such that 
\begin{align*}
w_1:\; &\bm_0^{(1)} \xrightarrow{\bd_1| \bd''_1}\bm_1^{(1)} \xrightarrow{\bd_2| \bd''_2} \bm_2^{(1)} \xrightarrow{\bd_3| \bd''_3} \cdots, \\
w_2:\; &\bm_0^{(2)} \xrightarrow{\bd_1''| \bd'_1} \bm_1^{(2)} \xrightarrow{\bd_2''| \bd'_2}\bm_2^{(2)} \xrightarrow{\bd_3''| \bd'_3} \cdots
\end{align*}
are walks in $\Gamma_{S\cup \{\mathbf{0}\}}$ satisfying $\mathop{cdeg}(w_1)<\mathop{cdeg}(w)$ and $\mathop{cdeg}(w_2)=1$ ({\it i.e.}, $w_2$ is a walk in $\Gamma_R$), and $\bm_k=\bm_k^{(1)} + \bm_k^{(2)}$ for each $k\in \N$.
\end{proposition}

\begin{proof}
Fix $k\in \mathbb{N}$ and represent the vertex $\bm_k$ of $w$ as
\begin{equation}\label{eq:RsubdivisionK}
\bm_k= \bm^{(1)}_k + \bm^{(2)}_k  \qquad 
\big(\mathop{cdeg}(\bm^{(1)}_k) < \mathop{cdeg}(w) \text{ and }\mathop{cdeg}(\bm^{(2)}_k) = 1
\big).
\end{equation}
We use this representation of $\bm_k$ to construct the edges $\bm_{k-1}^{(1)} \xrightarrow{\bd_k| \bd''_k}\bm_{k}^{(1)}$ and $\bm_{k-1}^{(2)} \xrightarrow{\bd''_k| \bd'_k}\bm_k^{(2)}$ in $\Gamma_{\Z^d}$ in a way that $\bm_{k-1}=\bm^{(1)}_{k-1} + \bm^{(2)}_{k-1}$ with $\mathop{cdeg}(\bm^{(1)}_{k-1}) < \mathop{cdeg}(w) \text{ and }\mathop{cdeg}(\bm^{(2)}_{k-1}) = 1$.
 
Since $\mathcal{D}$ is a complete set of coset representatives of $\Z^d/M\Z^d$, there is a unique pair $(\bm_{k-1}^{(1)},\bd_k'')\in \Z^d \times \mathcal{D}$ such that 
\begin{equation}\label{eq:edgedegminusone}
\bm_{k}^{(1)} + \bd_k=M\bm_{k-1}^{(1)} + \bd''_k.
\end{equation} 
and, hence, a unique pair  $(\bm_{k-1}^{(2)},\bb_k)\in \Z^d \times \mathcal{D}$ such that 
\begin{equation}\label{eq:edgedegone}
\bm_{k}^{(2)}+ \bd''_k=M\bm_{k-1}^{(2)} + \bb_k.
\end{equation} 
Adding \eqref{eq:edgedegminusone} and \eqref{eq:edgedegone} this yields
\begin{equation}\label{eq:edgeSum}
\bm_{k}+\bd_k=M(\bm_{k-1}^{(1)} + \bm_{k-1}^{(2)}) + \bb_k.
\end{equation} 
The existence of the edge $\bm_{k-1} \xrightarrow{\bd_k| \bd'_k}\bm_{k}$  is equivalent to 
\begin{equation*}
\bm_{k}+\bd_k=M\bm_{k-1} + \bd_k'.
\end{equation*} 
Comparing this with \eqref{eq:edgeSum} and observing that $\mathcal{D}$ is a complete set of coset representatives of $\Z^d/M\Z^d$, yields that $\bm_{k-1}=\bm_{k-1}^{(1)} + \bm_{k-1}^{(2)}$ and $\bb_k=\bd_k'$. Thus \eqref{eq:edgedegminusone} and \eqref{eq:edgedegone} yield the edges
\[
\bm_{k-1}^{(1)} \xrightarrow{\bd_k| \bd''_k}\bm_{k}^{(1)}\in\Gamma_{\Z^d} \quad\hbox{and}\quad \bm_{k-1}^{(2)} \xrightarrow{\bd''_k| \bd'_k}\bm_k^{(2)}\in\Gamma_{\Z^d}.
\]
Moreover, by Lemma~\ref{lem:mono}, we have $\mathop{cdeg}(\bm^{(1)}_{k-1}) \le \mathop{cdeg}(\bm^{(1)}_{k})< \mathop{cdeg}(w)$  and $\mathop{cdeg}(\bm^{(2)}_{k-1}) = 1$. The geometric interpretation of this construction  is illustrated in Figure~\ref{fig:FigSubdivision}.  
 
\begin{figure}[h]
\begin{tikzpicture}[]
  \pgftext{\includegraphics[width=0.4\textwidth]{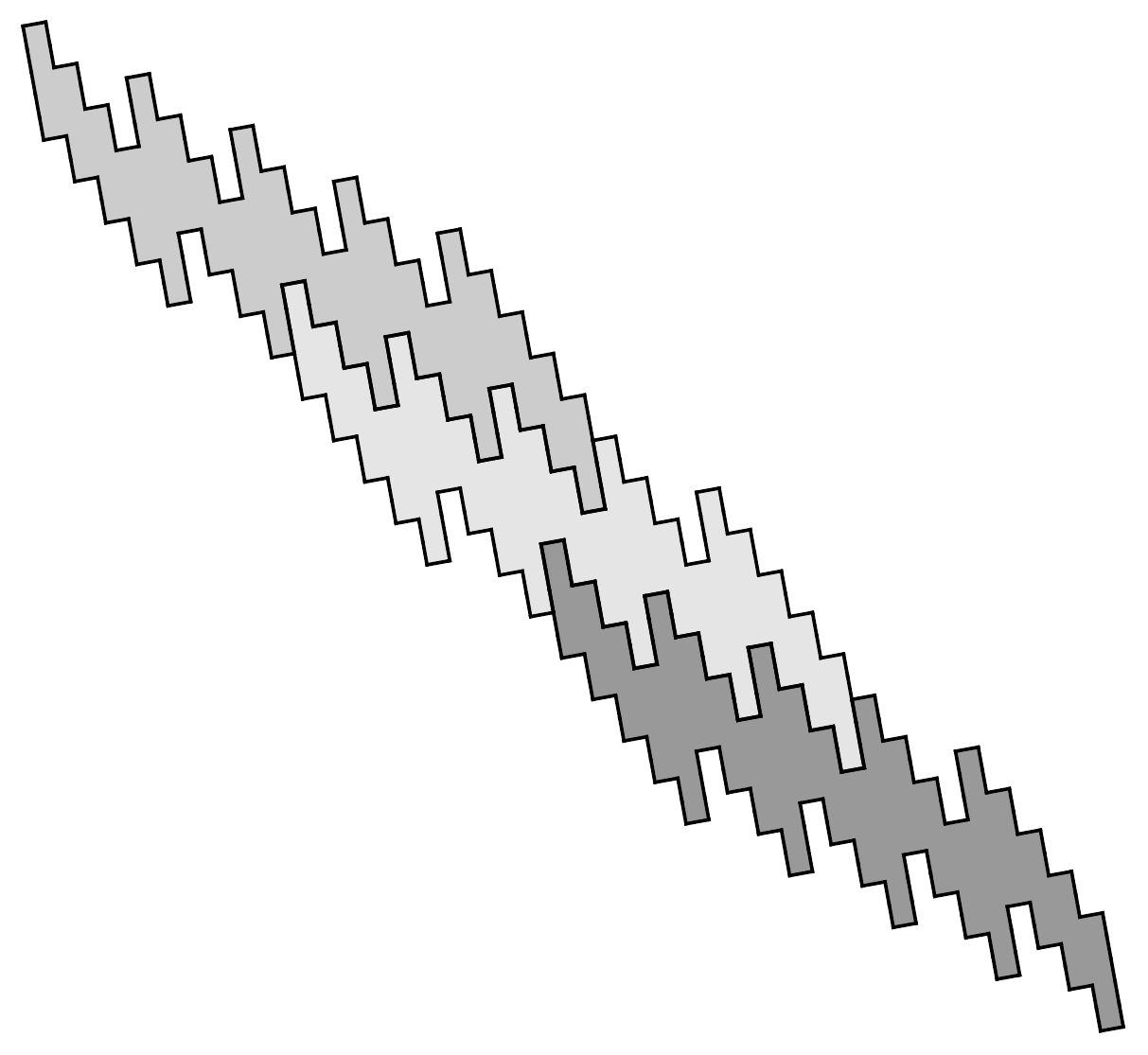}} at (0pt,0pt);
  \node at (2.6,-1.2) {$\mathcal{T}$};
  \node at (1.75,0.2) {$\mathcal{T}+\bm_{k}^{(1)}$};
  \node at (0.85,1.7) {$\mathcal{T}+\bm_{k}^{(1)}+\bm_{k}^{(2)}$};
\end{tikzpicture}
\\
$\Bigg\downarrow$ ${\bx\mapsto M^{-1}(\bx+\bd_k)}$
\\[10pt]
\begin{tikzpicture}[]
  \pgftext{\includegraphics[width=0.4\textwidth]{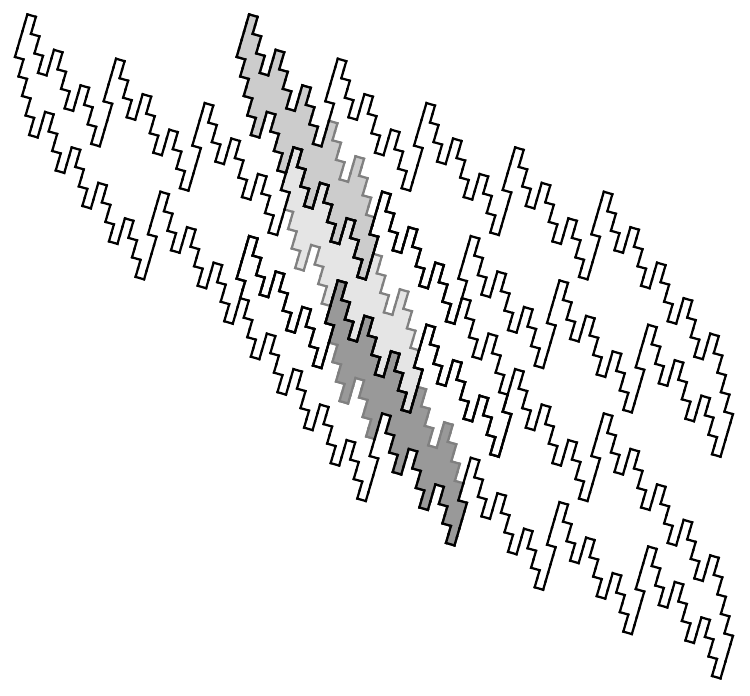}} at (0pt,0pt);
    \node at (0.5,-0.95) {\scriptsize $(1)$};
      \node at (0.15,0.1) {\scriptsize $(2)$};
        \node at (-0.21,1.15) {\scriptsize $(3)$};
          \node at (1,-2) {$\mathcal{T}$};
  \node at (-2.25,0.25) {$\mathcal{T}+\bm_{k-1}^{(1)}$};
  \node at (2,2) {$\mathcal{T}+\bm_{k-1}^{(1)}+\bm_{k-1}^{(2)}$};
\end{tikzpicture}
\caption{Geometric meaning of the representations of the nodes in the edge $\bm_{k-1}^{(1)}+\bm_{k-1}^{(2)}\xrightarrow{\bd_k|\bd'_k}\bm_{k}^{(1)}+\bm_{k}^{(2)}$. Here (1) is the subtile $M^{-1}(\cT+\bd_k)$ of $\cT$, (2) is the subtile $M^{-1}(\cT+\bm_{k}^{(1)}+\bd_k)$ of $\cT+\bm_{k-1}^{(1)}$ and (3) is the subtile $M^{-1}(\cT+\bm_{k}^{(1)}+\bm_{k}^{(2)}+\bd_k)=M^{-1}(\cT+M(\bm_{k-1}^{(1)}+\bm_{k-1}^{(2)})+\bd'_k)$ of $\cT+\bm_{k-1}^{(1)}+\bm_{k-1}^{(2)}$. One can see that if the small tiles are contact neighbors, then also the large tiles are contact neighbors (Drawn for the 3rd approximation.) \label{fig:FigSubdivision}}
\end{figure}

Iterating this we constructed walks 
\begin{align*}
w_{1,k}:\;&\bm_0^{(1)} \xrightarrow{\bd_1| \bd''_1}\bm_1^{(1)} \xrightarrow{\bd_2| \bd''_2} \bm_2^{(1)} \xrightarrow{\bd_3| \bd''_3} \cdots \xrightarrow{\bd_{k-1}| \bd''_{k-1}} \bm_{k}^{(1)},  \\
w_{2,k}:\;&\bm_0^{(2)} \xrightarrow{\bd_1''| \bd'_1} \bm_1^{(2)} \xrightarrow{\bd_2''| \bd'_2}\bm_2^{(2)} \xrightarrow{\bd_3''| \bd'_3} \cdots  \xrightarrow{\bd_{k-1}''| \bd'_{k-1}} \bm_{k}^{(2)}
\end{align*}
in $\Gamma_{\Z^d}$ satisfying $\mathop{cdeg}(w_{1,k})<\mathop{cdeg}(w)$ and $\mathop{cdeg}(w_{2,k})=1$, and $\bm_i=\bm_i^{(1)} + \bm_i^{(2)}$ for each $i\le k$.

Since $\mathop{cdeg}({\bm_i})$ is bounded by $p$, a given $\bm_i$ can have only finitely many representations \eqref{eq:RsubdivisionK}, the bound only depending on $R$ and $p$. Thus, by a Cantor diagonal argument we may choose representations for all $\bm_i$, $i\in \mathbb{N}$, in a way that the representation of $\bm_{i-1}$ is determined by the representation of $\bm_i$ for all $i\in\N$ in the way described above. 
This leads to walks $w_1,w_2 \in \Gamma_{\Z^d}$ satisfying $\mathop{cdeg}(w_1)<\mathop{cdeg}(w)$ and $\mathop{cdeg}(w_2)=1$, and $\bm_k=\bm_k^{(1)} + \bm_k^{(2)}$ for each $k\in \N$.
 
 Because $R$ is a finite set, there are only finitely many elements of $\Z^d$ with bounded contact degree. Thus $w_1,w_2 \in \Gamma_{\Z^d}$ contain infinitely many loops. Hence, by Lemma~\ref{lem:tileendsinloop}, we have $w_1,w_2\in \Gamma_{S\cup \{\mathbf{0}\}}$, and the result is proved.
\end{proof}

\begin{theorem}
Let $\cT=\cT(M,\mathcal{D})$ be a self-affine tile with standard digit set.  Then Algorithm~\ref{alg:Tileneighbor} terminates after finitely many steps and has the neighbor graph $\Gamma_S$ of $\cT$ as its output.
\end{theorem} 

\begin{proof}
Let $\Gamma_{R_1}=\Gamma_{R}$ and $\Gamma_{R_q}=\mathop{Red}(\Gamma_{R_{q-1}+R})$
We prove by induction that $\Gamma_{R_q}$ contains all infinite walks $w$ of $\Gamma_S$ having $\mathop{cdeg}(w) \le q$. 

Let $w$ be an infinite walk in $\Gamma_{S\cup \{\mathbf{0}\}}$ with $\mathop{cdeg}(w) = 1$ then, by definition, $w$ is a walk whose nodes are contained in $R$. Thus $w$ is a walk in $\Gamma_R$ and, hence, in $\Gamma_{R_1}$. This constitutes the induction start.

To prove the induction step we assume that $\Gamma_{R_{q-1}}$ consists of all infinite walks of $\Gamma_{S\cup \{\mathbf{0}\}}$ having contact degree less than or equal to $q-1$. Let $w:\; \bm_0  \xrightarrow{\bd_1| \bd'_1} \bm_1  \xrightarrow{\bd_2| \bd'_2}\bm_2  \xrightarrow{\bd_3| \bd'_3} \cdots$ be an infinite walk in $\Gamma_{S\cup \{\mathbf{0}\}}$ having $\mathop{cdeg}(w) \le q$. Then, by Proposition~\ref{lem:2} and by the induction hypothesis we know that there exist infinite walks $w_1:\;\bm_0^{(1)} \xrightarrow{\bd_1| \bd''_1} \bm_1^{(1)}  \xrightarrow{\bd_2| \bd''_2} \bm_2^{(1)}  \xrightarrow{\bd_3| \bd''_3} \cdots
\in \Gamma_{R_{q-1}}$ and $w_2:\;\bm_0^{(2)} \xrightarrow{\bd_1''| \bd'_1} \bm_1^{(2)}  \xrightarrow{\bd_2''| \bd'_2}\bm_2^{(2)}  \xrightarrow{\bd_3''| \bd'_3} \cdots \in \Gamma_{R_{1}}$ such that $\bm_k=\bm_k^{(1)}+\bm_k^{(2)}$ for each $k\in \N$. Because $k$ was arbitrary, $w$ is an infinite walk in $\Gamma_{R_{q-1}+R}$ and, therefore, an infinite walk in $\Gamma_{R_q}=\mathop{Red}(\Gamma_{R_{q-1}+R})$. 

Since there is $p\in\N$ such that each walk $w$ in $\Gamma_S$ satisfies $\mathop{cdeg}(w) \le p$, we conclude that $\Gamma_S=\Gamma_{R_{p+1}}=\Gamma_{R_{p}}$. Thus the algorithm terminates after finitely many steps and returns $\Gamma_S$, as desired.
\end{proof}

\section{Pisot substitutions, Rauzy fractals, and self-replicating tilings}\label{sec:RauzyBasic}

\subsection{Substitutions}
For $d\ge 2$ we fix the alphabet $\cA=\{1,\ldots, d\}$ and write $\cA^*$ for the free monoid of finite words over $\cA$ with concatenation.  For $w\in \cA^*$ we denote by $|w|$ the number of letters in $w$ and call it the {\em length} of $w$. The word of length $0$ is called the {\em empty word} and we denote it by $\epsilon$. For $i\in \cA$, we write $|w|_i$ for the number of occurrences of the letter $i$ in $w$. We define the {\em abelianization mapping} $\mathbf{l}:\cA^* \to \N^d$ by $\mathbf{l}(w)=(|w|_i)_{i\in\cA}$. 

A {\em substitution} is an endomorphism of $\cA^*$ having the property that each letter has a nonempty image. To a substitution $\sigma$ we associate the {\em incidence matrix} $M$ defined as the unique matrix $M$ satisfying  $\mathbf{l}(\sigma(i))=M\mathbf{l}(i)$ for each $i\in\cA$. If the characteristic polynomial of $M$ is the minimal polynomial of a Pisot unit, we call $\sigma$ a {\it Pisot substitution}.  

\subsection{The Rauzy fractal and its subtiles} We give a definition of the Rauzy fractal related to a Pisot substitution. We do this very briefly, for more details on Rauzy fractals we refer for instance to \cite{Arnoux-Ito:01,CANTBST,Ito-Rao:06,SirventWang02}. 

\begin{definition}[Prefix-suffix graph; {{\em cf.~e.g.}~\cite{Canterini-Siegel:01a}}]\label{def:presufgraph}
Let $\sigma$ be a substitution over the alphabet $\cA$ and let
\[
\cP= \big\{(p,i,s)\in \cA^*\times \cA \times \cA^*  \;:\;  \text{there is }j\in\cA\text{ with }\sigma(j)=pis\big\}.
\]
The {\em prefix-suffix graph} of $\sigma$ is the directed labeled graph whose vertices are the elements of $\cA$ with a directed edge from $i\in\cA$ to $j\in\cA$ labeled by $(p,i,s)\in \cP$ if and only if $\sigma(j)=pis$. 
\end{definition}

\begin{example}[Prefix-suffix graph] 
Let $\sigma_1$ and $\sigma_2$ be  substitutions over $\cA=\{1,2,3\}$ given in the following way.
$$\sigma_1: \left\{\begin{aligned}
1&\longrightarrow 1112\\
2&\longrightarrow 113\\
3&\longrightarrow 1\\
\end{aligned}
\right.
\quad \quad\quad\quad
\sigma_2: \left\{\begin{aligned}
1&\longrightarrow 112\\
2&\longrightarrow 1113\\
3&\longrightarrow 1\\
\end{aligned}
\right.
$$ 
Definition~\ref{def:presufgraph} can now be applied to $\sigma_1$ and $\sigma_2$ to construct the prefix-suffix graphs depicted in Figure~\ref{Pre-Sufgraph}.
\begin{figure}[h] 
\hskip 0.1cm \xymatrix{
 *++[o][F]{1}  
\ar@(l,u)[]^{(\epsilon,1,112),(1,1,12),(11,1,2)}\ar@/^{0ex}/[rrr]_{(\epsilon,1,13),(1,1,3)}\ar@/^{-3ex}/[drr]^{(\epsilon,1,\epsilon)}& 
& &*++[o][F]{2}\ar@/_{5ex}/[lll]_{(111,2,\epsilon)}\\
&&*++[o][F]{3}\ar@/^{-3ex}/[ru]^{(11,3,\epsilon)} &
}
\xymatrix{
 *++[o][F]{1}  
\ar@(l,u)[]^{(\epsilon,1,12),(1,1,2)}\ar@/^{0ex}/[rrr]_{(\epsilon,1,113),(1,1,13),(11,1,3)}\ar@/^{-3ex}/[drr]^{(\epsilon,1,\epsilon)}& 
& &*++[o][F]{2}\ar@/_{5ex}/[lll]_{(11,2,\epsilon)}\\
&&*++[o][F]{3}\ar@/^{-3ex}/[ru]^{(111,3,\epsilon)} &
}

\caption{Prefix-suffix graph for $\sigma_1$ (left) and $\sigma_2$ (right).}\label{Pre-Sufgraph}
\end{figure}
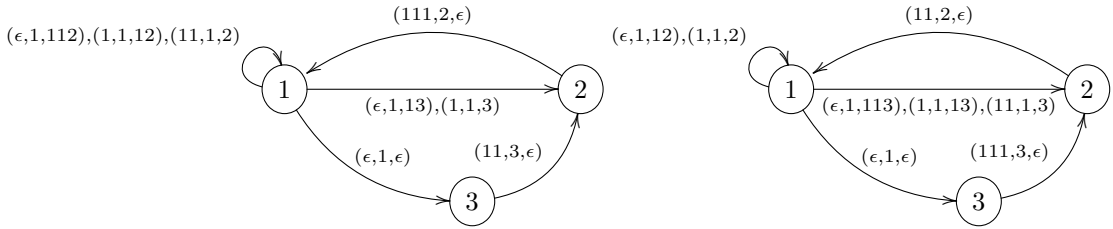
It is easy to check that $\sigma_1 $ and $\sigma_2$ are both Pisot substitutions. Indeed, the characteristic equations for their incidence matrices are given by $x^3-3x^2-2x-1=0$ and $x^3-2x^2-3x-1=0$, respectively. Both of these polynomials have a Pisot unit as dominant root.

We will use the substitutions $\sigma_1$ and $\sigma_2$ throughout the paper as our running examples.
\end{example}

In the sequel, for the orthogonal complement of a vector $\bw\in\R^d$ we write $\bw^\bot$. Let $\sigma$ be a Pisot substitution with incidence matrix $M$ and let $\bu$ and $\bv$ be the dominant right and left eigenvector of $M$, respectively. The matrix $M$ leaves the space $\bv^\bot$ invariant and we write $\mathbf{h}=M|_{\bv^\bot}$ for the restriction of the linear mapping $M$ to $\bv^\bot$. Then the linear mapping $\mathbf{h}$ is a uniform contraction on $\bv^{\bot}$. We write $\pi:\, \R^d\to \bv^\bot$ for the projection along $\bu$ onto $\bv^\bot$ .

Using these preparations we define Rauzy fractals and their subtiles as follows (see for instance \cite[Section~4]{SirventWang02} or \cite[Theorem~5.26]{CANTBST}).

\begin{definition}[Rauzy fractal and its subtiles]\label{def:Rauzy}
Let $\sigma$ be a Pisot substitution over the alphabet~$\cA$. Then the {\em Rauzy fractal} $\cR$ of $\sigma$ is defined as $\cR=\bigcup_{i\in \cA} \cR(i)$, where the {\em subtiles} $\cR(i)$ are the nonempty compact sets that are uniquely defined by the graph-directed iterated function system
\begin{equation}\label{eq:RauzySetEquation}
\cR(i) = \bigcup_{\begin{subarray}{c}  j\in \cA \\ i \xrightarrow{(p,i,s)}j \end{subarray}}
\mathbf{h} \cR(j) + \pi\mathbf{l}(p)
\qquad{(i\in \cA)}.
\end{equation}
The union is extended over all edges of the prefix-suffix graph of $\sigma$ leading away from the vertex~$i$.
\end{definition}

Note that by this definition the Rauzy fractal of a Pisot substitution $\sigma$ is a subset of $\bv^\bot$, where $\bv$ is the dominant left eigenvector of the incidence matrix of $\sigma$.

\begin{example}
For $\sigma_1$, by Figure~\ref{Pre-Sufgraph}, the set equations in \eqref{eq:RauzySetEquation} read
\begin{equation*}
\begin{aligned}
\cR(1)&=\mathbf{h} \cR(1)\cup \big(\mathbf{h} \cR(1)+\pi\mathbf{l}(1)\big)\cup \big(\mathbf{h} \cR(1)+\pi\mathbf{l}(11)\big)\cup \mathbf{h} \cR(2)\cup \big(\mathbf{h} \cR(2)+\pi\mathbf{l}(1)\big)\cup\mathbf{h} \cR(3),\\
\cR(2)&=\mathbf{h} \cR(1)+\pi\mathbf{l}(111),\\
\cR(3)&=\mathbf{h} \cR(2)+\pi\mathbf{l}(11).\\
\end{aligned}
\end{equation*}
Their solution corresponds to the Rauzy fractal drawn on the left-hand side of  Figure \ref{Sigma_Rauzy}.  
The set equations for the subtiles associated with $\sigma_2$ can be set up similarly. The corresponding Rauzy fractal is drawn on the right-hand side of  Figure~\ref{Sigma_Rauzy}. 

\begin{figure}[h]
\includegraphics[width=4.5 cm]{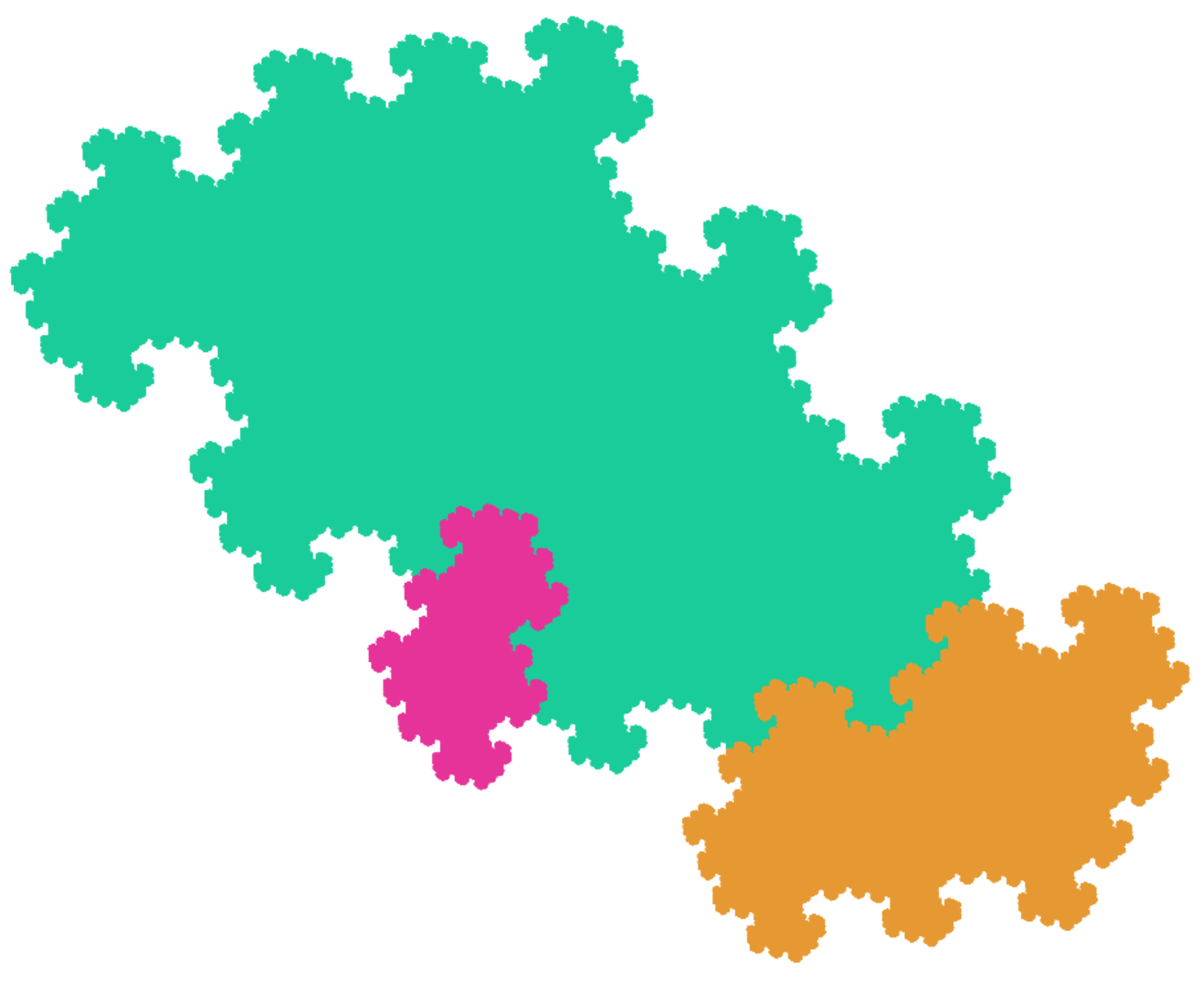} \quad\quad\quad\quad  \includegraphics[width= 4.5 cm]{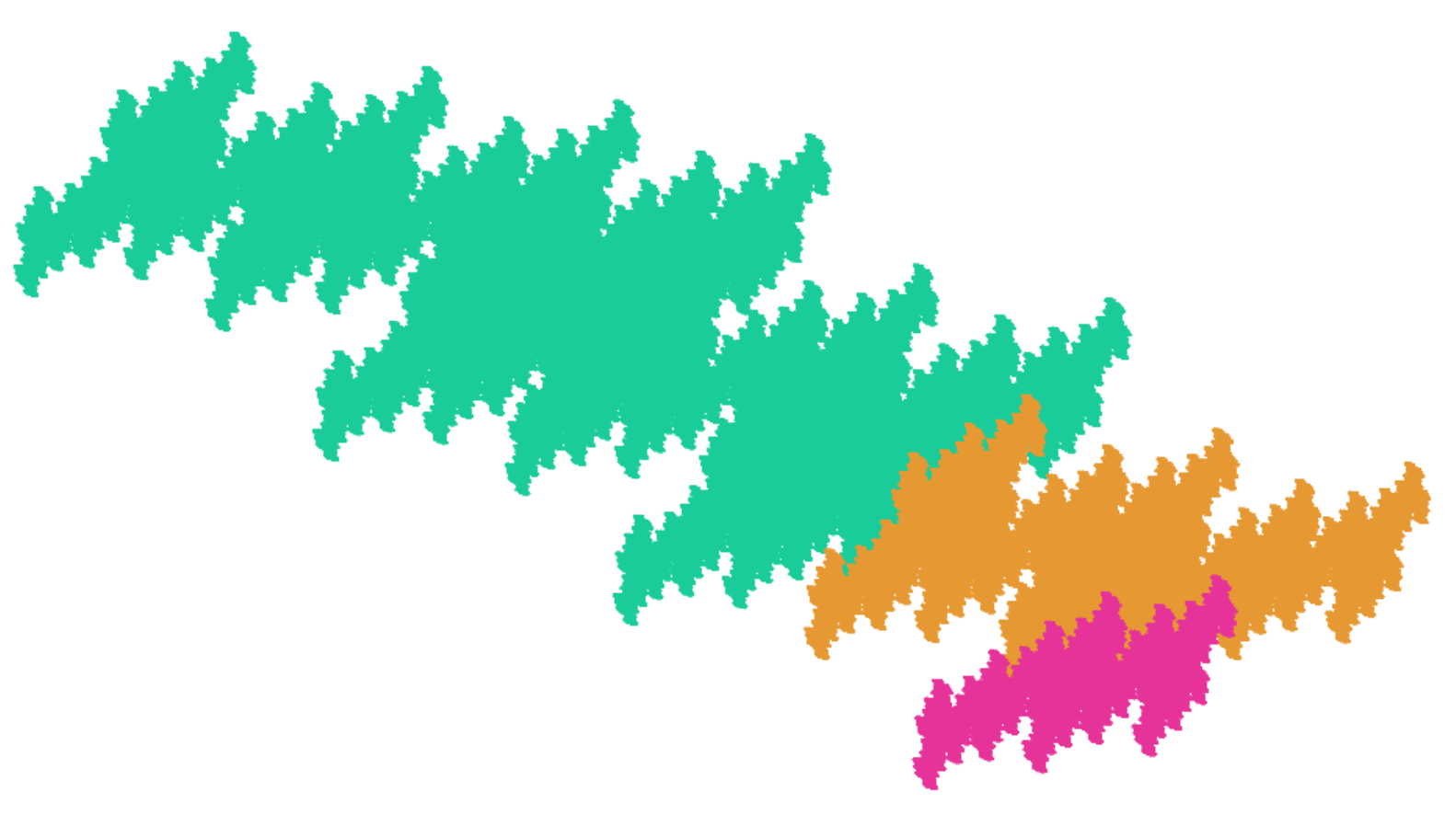}
\caption{Rauzy fractals and their subtiles associated with $\sigma_1$ (left) and $\sigma_2$ (right).}\label{Sigma_Rauzy}
\end{figure}
\end{example}

\subsection{The self-replicating tiling set and dual substitutions}
It turns out that subtiles of Rauzy fractals induce multi-tilings. These multi-tilings will have a self-replicating structure that is governed by the underlying substitution. In particular, let $\sigma$ be a Pisot substitution over the alphabet $\cA$ and use the notations from above. Let $\be_1,\ldots, \be_d$ be the standard basis vectors. The {\em stepped hypersurface} associated to $\sigma$ is given by
\[
H_\sigma = \big\{[\bx,i] \in \Z^d \times \cA \;:\; 
0 \le \langle \bx,\bv \rangle <  \langle \be_i,\bv \rangle
\big\}
\]
(see {\em e.g.} \cite[Equation (2.2)]{Ito-Rao:06}; recall that $\bv$ is  the left eigenvector of the incidence matrix $M$ of $\sigma$). A pair $[\bx,i] \in \Z^d \times \cA$ is called a {\em face} and can be regarded as a point of ``color'' $i$. However, we can also interpret a face $[\bx,i]$ as the cube given by
\begin{equation}\label{eq:face=cube}
[\bx,i] = \{ x+\vartheta_1\be_1 + \cdots+ \vartheta_{i-1}\be_{i-1}+  \vartheta_{i+1}\be_{i+1} + \cdots +\vartheta_d\be_d\;:\; \vartheta_j\in [0,1] \text{ for } j\in\cA\setminus\{i\} \}.
\end{equation}
If we project the stepped hypersurface onto $\bv^\bot$ we get the {\it self-replicating translation set}
\[
\pi H_\sigma = \big\{[\pi(\bx),i]  \;:\; [\bx, i] \in H_\sigma
\big\}.
\]
The ``projected face'' $[\pi (\bx),i]$ can be regarded as the parallelotope
\[
[\pi(\bx),i] = \pi\big(\big\{\bx+ \vartheta_1\be_1 + \cdots+ \vartheta_{i-1}\be_{i-1}+  \vartheta_{i+1}\be_{i+1} + \cdots +\vartheta_d\be_d\;:\; \vartheta_j\in [0,1] \text{ for } j\in\cA\setminus\{i\} \big\}\big).
\]
These parallelotopes are analogs of the fundamental mesh of a lattice and its translates in the sense that they tesselate $\bv^{\bot}$ (see for instance~\cite{ABFJ:07,Berthe-Vuillon:00}).
\begin{example} 
In Figure \ref{StepSurface} we provide patches of the stepped surface of the substitutions $\sigma_1$ and $\sigma_2$. The ``white face'' in the figure shows squares of the form $[\bx,1]$, the ``grey face'' shows squares of the form $[\bx,2]$ and the ``black face''  shows squares of the form $[\bx,3]$.
\begin{figure}[h]
\includegraphics[width=4.5 cm]{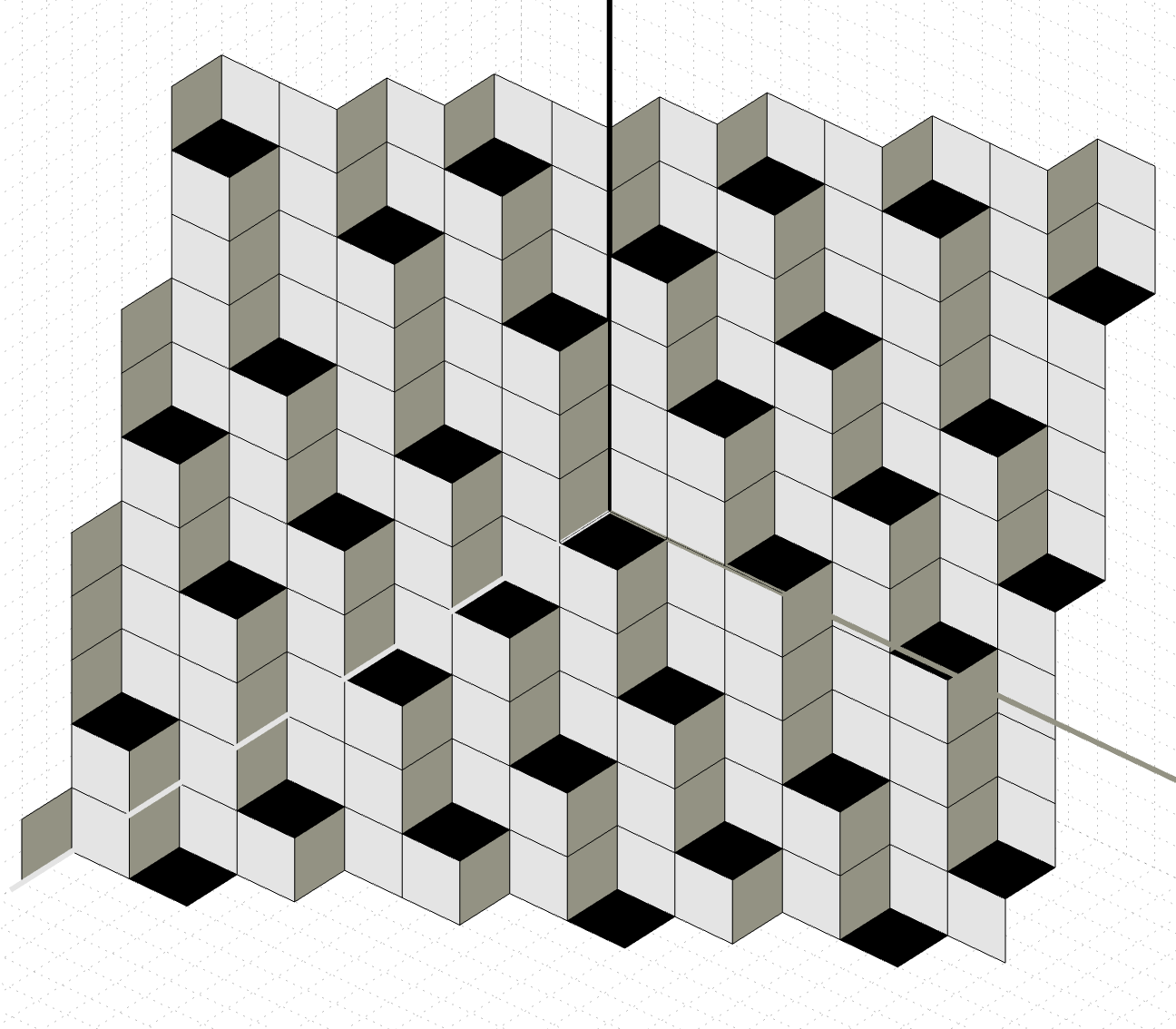} \quad\quad\quad\quad  \includegraphics[width= 4.5 cm]{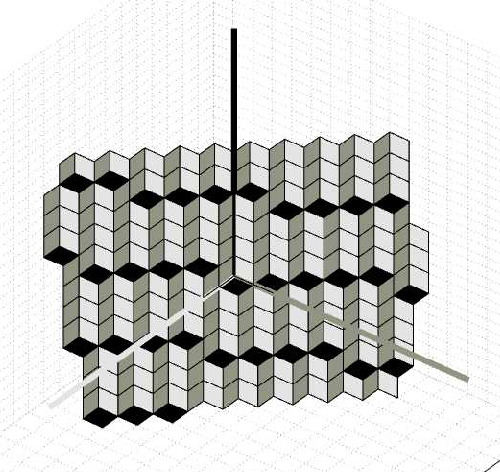}\\
\caption{The stepped surface for $\sigma_1$ (left) and $\sigma_2$ (right), respectively.
} \label{StepSurface}
\end{figure}
\end{example}

Let $\sigma$ be a Pisot substitution. In \cite[Section~2]{Arnoux-Ito:01} a ``dual'' of $\sigma$, acting on subsets of $\mathbb{Z}^d\times\cA$ was defined. Following \cite[Section~2.4]{Ito-Rao:06} we define this dual in the following way. For a singleton $\{[\bx,i]\}\subset \mathbb{Z}^d\times\cA$ let
\[ 
\sigma^*\{[\bx,i]\}=
\big\{[M^{-1}(\bx+\mathbf{l}(p)),j]\;:\; i\xrightarrow{(p,i,s)}j \text{ is an edge in the prefix-suffix graph of $\sigma$} \big\},
\]
and for $K\subset \mathbb{Z}^d\times\cA$ put
\[
\sigma^*(K) =\bigcup_{[\bx,i]\in K}\sigma^*\{[\bx,i]\}.
\]
For the sake of simplicity we will often write $\sigma^*[\bx,i]$ instead of $\sigma^*\{[\bx,i]\}$. It follows from the definition that $(\sigma^*)^n = (\sigma^n)^*$, {\it i.e.}, that the $n$-th iteration of the dual of $\sigma$ is the dual of the $n$-th iteration of $\sigma$. Moreover, it was shown in \cite[Section~3]{Arnoux-Ito:01} (see also \cite[Theorem~1.5]{EiItoRao06}) that $H_\sigma$ is invariant under $\sigma^*$ in the following sense.

\begin{lemma}\label{lem:AItiling}
Let $\sigma$ be a Pisot substitution. Then the following assertions hold.
\begin{itemize}
\item[(i)] $\sigma^*(H_\sigma)=H_\sigma$.
\item[(ii)] If $[\bx,i],[\by,j]\in H_\sigma$ are distinct, then $\sigma^*[\bx,i] \cap \sigma^*[\by,j] = \emptyset$.
\end{itemize}
\end{lemma}
 \begin{figure}[h]
\includegraphics[width=3.5 cm]{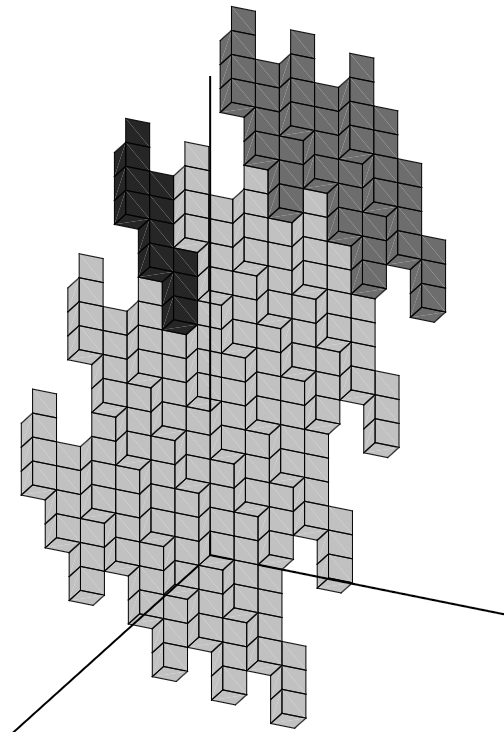} \quad\quad\quad\quad  \includegraphics[width= 3 cm]{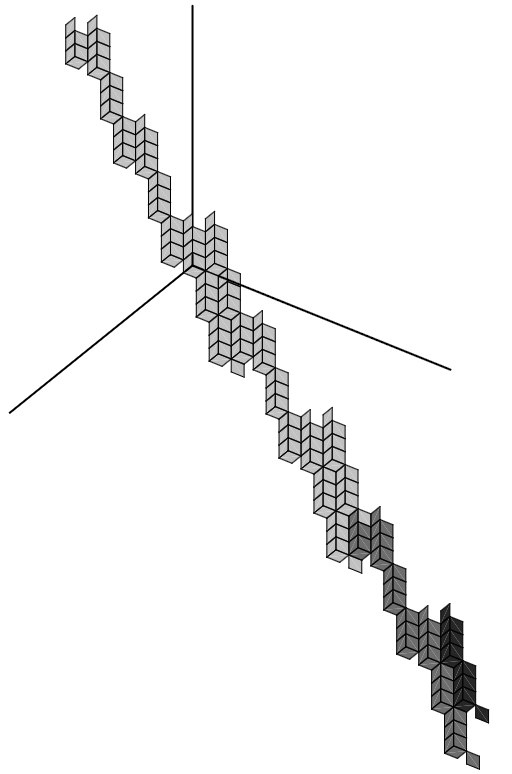}\\
\caption{The $(\sigma^*)^n([\mathbf{0},i]) (i=1,2,3)$ for $\sigma_1$ (left) and $\sigma_2$ (right) with $n=4$, respectively. To get the $\cR_n(i)$ is to apply the projection $\pi$ and uniformly contraction $h^n$.
} \label{Dual_approxi}
\end{figure}
 The set equation \eqref{eq:RauzySetEquation} and its iterations immediately imply that
 \[
 \cR(i) = \bigcup_{[\by,j]\in\sigma^*[\mathbf{0},i]} \mathbf{h}(\cR(j) +\pi\by)=
  \bigcup_{[\by,j]\in(\sigma^*)^n[\mathbf{0},i]} \mathbf{h}^n(\cR(j) +\pi\by) \qquad (n\in\N).
 \]
If we interpret the faces geometrically as in \eqref{eq:face=cube}, we can regard 
\begin{equation}\label{eq:rapprox}
\cR_n(i)=\mathbf{h}^n\pi(\sigma^*)^n[\mathbf{0},i]
\qquad (i\in\cA,\, n\in\N)
\end{equation}
as $n$-th approximation of $\cR(i)$ consisting of finitely many $(d-1)$-dimensional parallelotopes of the form $\mathbf{h}^n[\pi(\bx),i]$ with $[\bx,i]\in H_\sigma$. More precisely, we have $\lim_{n\to\infty} \cR_n(i)=\cR(i)$, $i\in\cA$, in Hausdorff metric; see~\cite[Proposition~3.2]{Ito-Rao:06}. Lemma~\ref{lem:AItiling} implies the following result.
 
 \begin{lemma}\label{lem:nthtiling}
 Let $\sigma$ be a Pisot substitution and $n\in\N$. The following assertions hold.
 \begin{itemize}
 \item[(i)] $\mathcal{E}_n =\{(\sigma^*)^n[\bx,i] \colon [\bx,i]\in H_\sigma \}$ is a partition of $H_\sigma$.
 \item[(ii)] $\mathcal{I}_n =\{\cR_n(i) + \pi\bx  \colon [\bx,i]\in H_\sigma \}$ is a tiling of $\bv^\bot$.
 \end{itemize}
 \end{lemma}
 
 \begin{proof}
Both assertions follow immediately by applying Lemma~\ref{lem:AItiling} with $\sigma$ replaced by the Pisot substitution $\sigma^n$.
 \end{proof}
  
 Lemma~\ref{lem:nthtiling}~(i) provides a tiling of a discrete set. However, if we interpret the faces geometrically as in \eqref{eq:face=cube}, we can regard it also as a tiling of the stepped hypersurface represented by $H_\sigma$. The ``tiles'' $(\sigma^*)^n[\mathbf{x},i]$, $[\bx,i]\in H_\sigma$, are then finite unions of cubes of the form \eqref{eq:face=cube}.

 \begin{figure}[h]
\includegraphics[width=5 cm]{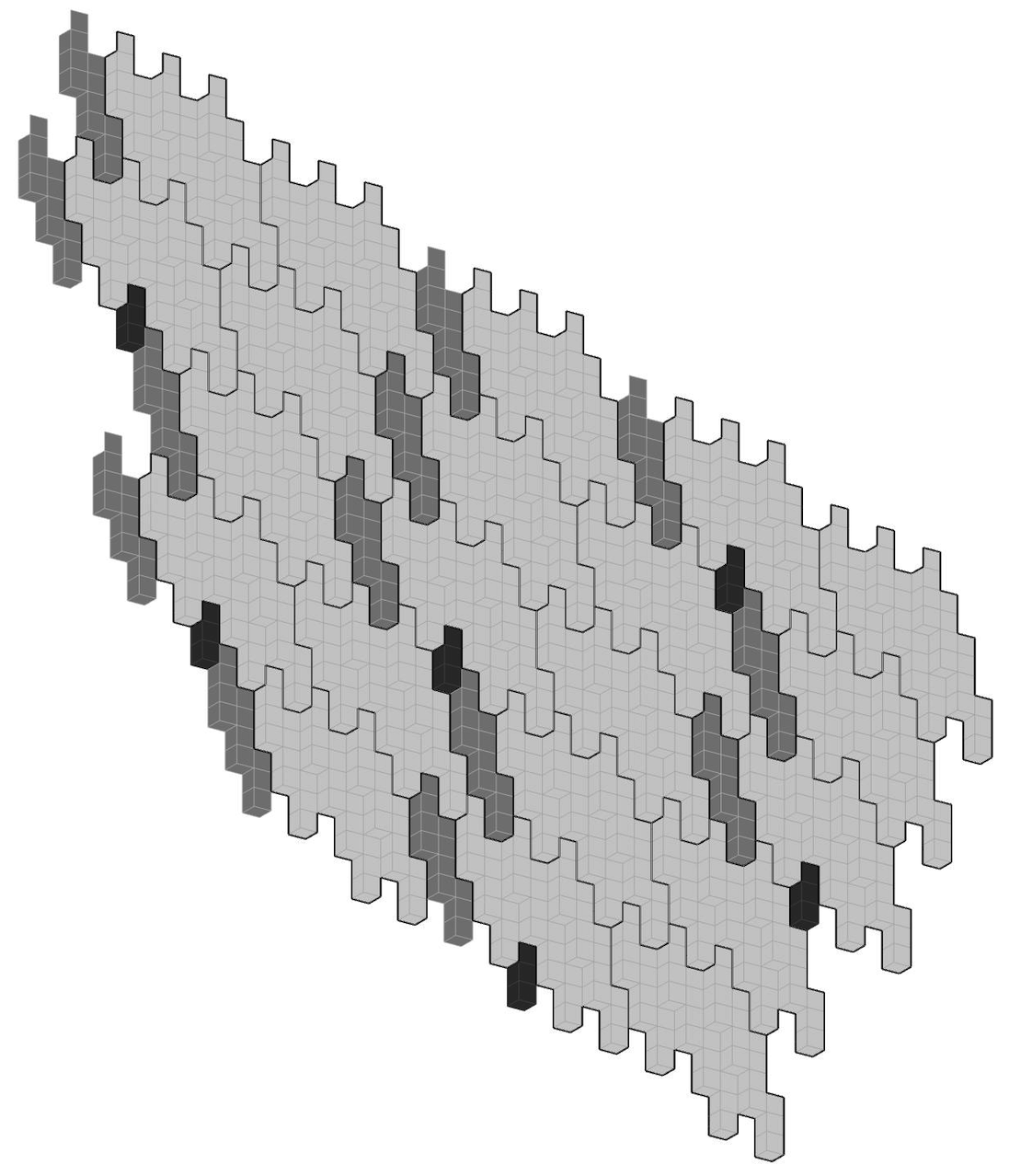} 
\caption{A patch of tiling  shows that $(\sigma^*)^n([x,i])$ tiles the stepped hypersurface with $[x,i] \in H_{\sigma}$ where $n=3$ and  $\sigma=\sigma_1$.} \label{TilingPatch}
\end{figure}

It is known that  
\[
\mathcal{C}_0=\{ \cR(i) + \pi(\bx) \,:\, [\bx,i] \in H_\sigma \}
\]
forms a multi-tiling of $\bv^\bot$, which is called the {\em self-replicating multi-tiling} ({\it cf.~e.g.}~\cite{CANTBST}). As mentioned in the introduction, the question whether $\mathcal{C}_0$ even forms a tiling of $\bv^\bot$ corresponds to the famous {\em Pisot conjecture} (see~\cite{ABBLS:15} and the references given there).


\section{Graphs related to the boundary of a Rauzy fractal}\label{sec:RGRAPHS}

\subsection{Neighbors of Rauzy fractals}\label{sec:RauzyNeighbors}
Let $\sigma$ be a Pisot substitution. We are interested in {\em neighbors} of the subtiles $\cR(i)$, $i\in \cA$, in the self-replicating multi-tiling $\mathcal{C}_0$. In other words, we want to find all faces $[\bx,j] \in H_\sigma$ having $\cR(i) \cap (\cR(j) + \pi(\bx)) \not=\emptyset$ for some $i\in\cA$. To achieve this, in \cite[Section~5]{ST:09} certain graphs are defined. Before we recall their definition we want to convey the idea of their construction. \\
We will use the following notation. Given an element $[i,\bx,j]\in \cA\times \Z^d \times \cA$ we set $$-[i,\bx,j]=[j,-\bx,i]$$ and write $\pm K = K \cup (-K)$ for $K\subset \cA \times \Z^d\times \cA$.\\
By the defining set equation \eqref{eq:RauzySetEquation} for subtiles, for each $[i,\bx,j]\in\cA \times H_\sigma$ we may write
\[
\begin{split}
\cR(i) \cap (\cR(j) + \pi(\bx)) &= \bigcup_{\begin{subarray}{c} \sigma(i_1)=p_1is_1 \\ \sigma(j_1)=q_1jt_1\end{subarray}}
\mathbf{h}
\Big(
\big(\cR(i_1) +  \mathbf{h}^{-1}\pi\mathbf{l}(p_1)\big)\cap \big(
\cR(j_1) + \mathbf{h}^{-1}\pi(\mathbf{l}(q_1)+\bx)
\big)
\Big) 
\\
&= \bigcup_{\begin{subarray}{c} \sigma(i_1)=p_1is_1 \\ \sigma(j_1)=q_1jt_1\end{subarray}}
\mathbf{h}
\Big(
\cR(i_1) \cap \big(
\cR(j_1) + \mathbf{h}^{-1}\pi(\mathbf{l}(q_1) -\mathbf{l}(p_1)+\bx)
\big)
\Big) + \pi\mathbf{l}(p_1).
\end{split}
\]
If we set $\cB[i,\bx,j] = \cR(i) \cap (\cR(j) + \pi(\bx)) \not=\emptyset$ this implies that
\begin{equation}\label{eq:seteqBixj}
\cB[i,\bx,j] = \bigcup_{\begin{subarray}{c} \sigma(i_1)=p_1is_1 \\ \sigma(j_1)=q_1jt_1\end{subarray}}
\mathbf{h} \cB[i_1,M^{-1}(\mathbf{l}(q_1) - \mathbf{l}(p_1) + \bx),j_1] + \pi\mathbf{l}(p_1).
\end{equation}
Thus the nonempty intersections $\cB[i,\bx,j]$ are solutions of a graph directed iterated function system. Because $[M^{-1}\mathbf{l}(p_1),i_1]\in\sigma^*[\mathbf{0},i]$ and $[M^{-1}(\bx+\mathbf{l}(q_1)),j_1]\in\sigma^*[\bx,j]$, we know from Lemma~\ref{lem:AItiling} that $[M^{-1}\mathbf{l}(p_1),i_1],[M^{-1}(\bx+\mathbf{l}(q_1)),j_1]\in H_\sigma$. According to the definition of $H_\sigma$ this implies that $[i_1,M^{-1}(\mathbf{l}(q_1) - \mathbf{l}(p_1) + \bx),j_1]\in \pm(\cA\times H_\sigma)$.

We have two options in the set equation \eqref{eq:seteqBixj}. In the first option, we want to deal exclusively with elements of $\cA\times H_\sigma$. Thus, whenever $[M^{-1}(\mathbf{l}(q_1) - \mathbf{l}(p_1) + \bx),j_1]\not\in\cA\times H_\sigma$, using the identity
\[
\mathbf{h}\cB[i_1,M^{-1}(\mathbf{l}(q_1) - \mathbf{l}(p_1) + \bx),j_1] + \pi\mathbf{l}(p_1)=
\mathbf{h}\cB[j_1,-M^{-1}(\mathbf{l}(q_1) - \mathbf{l}(p_1) + \bx),i_1] + \bx+\pi\mathbf{l}(q_1)
\]
we replace $\mathbf{h}\cB[i_1,M^{-1}(\mathbf{l}(q_1) - \mathbf{l}(p_1) + \bx),j_1] + \pi\mathbf{l}(p_1)$ by $\mathbf{h}\cB[j_1,-M^{-1}(\mathbf{l}(q_1) - \mathbf{l}(p_1) + \bx),i_1] + \bx+\pi\mathbf{l}(q_1)$ on the right hand side of the set equation \eqref{eq:seteqBixj}. Indeed, by this replacement we can even avoid that elements of the form $[i,\mathbf{0},j]$ with $i>j$ occur in \eqref{eq:seteqBixj}, because in this case we can replace $\mathbf{h}\cB[i,\mathbf{0},j]$ by $\mathbf{h}\cB[j,\mathbf{0},i]$. This entails that the triples $[i,\bx,j]$ occurring in this modified set equation belong to the set $\mathfrak{D}$ used in Definition~\ref{def:ambientG} below.

Secondly, we can leave the identity \eqref{eq:seteqBixj} as it is at the cost that the elements $[i,\bx,j]$ involved are contained in the larger set $\pm(\cA\times H_\sigma)$. 

For both of these options, in order to set up the graph directed iterated function system, we need to determine which of the sets $\cB[i,\bx,j]$ are nonempty. This is done with the help of the graphs that we will define in the Sections~\ref{sec:graph1} and~\ref{sec:graph2} below. The first option leads to smaller graphs for the graph directed iterated functions system and is therefore more convenient when dealing with examples. The second option makes the definition of the graphs easier and is better suited for theoretical considerations. For this reason, we will provide the details for both options.

\subsection{The self-replicating boundary graph and the contact graph}\label{sec:graph1}
We start with the first option indicated in Section~\ref{sec:RauzyNeighbors}.
By the considerations of Section~\ref{sec:RauzyNeighbors}, the nonempty intersections $\cB[i,\bx,j]$ are the solutions of a graph-directed iterated function system. The graph that governs this system is the {\it self-replicating boundary graph} whose definition we will now recall. We start with the definition of a large graph that will contain the graphs we are interested in.

\begin{definition}[Ambient graph, {\em cf.} {\cite[Section~5.2]{ST:09}}]
\label{def:ambientG}
Set
\[
\mathfrak{D}  = \big\{
[i,\bx,j]\in\cA\times H_\sigma\;:\;  \bx=\mathbf{0} \text{ implies } i< j 
\big\}.
\]
The {\em ambient graph} $G_{\mathfrak{D}}$ is the directed graph whose nodes are the elements of $\mathfrak{D}$. There is an edge $[i,\bx,j]\xrightarrow{\eta}[i',\bx',j']$ if there exist $((p_1,i,s_1),(q_1,j,t_1))\in\cP\times\cP$ with
\begin{equation*}
\text{(type 1) }\;
\begin{cases}
\sigma(i')= p_1 i s_1 \text{ and } \sigma(j')= q_1 j t_1, \\
M\bx' = \bx + \mathbf{l}(q_1) - \mathbf{l}(p_1),
\end{cases}
\text{or (type 2) }\; 
\begin{cases}
\sigma(j')= p_1 i s_1 \text{ and } \sigma(i')= q_1 j t_1, \\
-M\bx' = \bx + \mathbf{l}(q_1) - \mathbf{l}(p_1),
\end{cases}
\end{equation*}
where the label is given by 
\[
\eta=
\begin{cases}
\mathbf{l}(p_1)\,|\,\mathbf{l}(q_1), &  \langle \mathbf{l}(p_1), \bv \rangle \le  \langle \mathbf{l}(q_1) + \bx, \bv \rangle, \\
\mathbf{l}(q_1)\,|\,\mathbf{l}(p_1), &  \text{otherwise}.
\end{cases}
\]
\end{definition}

Note that, if $\bx=\mathbf{0}$, then $[i,\mathbf{0},j]$ and $[j,\mathbf{0},i]$ correspond to the same intersection and $[i,\mathbf{0},i]$ corresponds to the full subtile $\cR(i)$. Thus we excluded $[i,\mathbf{0},j]$ for $i\geq j$ from $\mathfrak{D}$ to avoid redundancies and trivialities. 

The following definition concerns the analog of the neighbor graph defined for self-affine tiles with standard digit sets (compare with the characterization of this neighbor graph in Lemma~\ref{lem:tileendsinloop}). In accordance with \cite{ST:09} we call it {\em self-replicating boundary graph} (rather than ``neighbor graph'').

\begin{definition}[Self-replicating boundary graph, {\em cf.} {\cite[Section~5.2.1]{ST:09}}]
\label{def:GSsrs}
The {\em self-replicating boundary graph} $G_B$ is the largest subgraph of the ambient graph for which each node belongs to a walk that ends in a loop. The set $B$ of nodes of $G_B$ is the {\em (self-replicating) neighbor set}.
\end{definition}

From the definition of the edges of the ambient graph it follows that an infinite walk can contain a loop only if it starts at an element $[i,\bx,j] \in \mathfrak{D}$ that satisfies
\[
\Vert\pi(\bx)\Vert \le 
 \frac{2 \max\{ \Vert\pi\mathbf{l}(p)\Vert\,:\, (p,a,s)\in \mathcal{P} \}}{1-\Vert \mathbf{h} \Vert}
 \]
 (see also \cite[Section~5.2]{ST:09}). Therefore, $G_B$ is contained in a finite subgraph of the ambient graph and, hence, it can be constructed algorithmically. The graph $G_B$ characterizes the intersections $\cR(i) \cap (\cR(j) + \pi(\bx))$ for $[i,\bx,j]\in \mathfrak{D}$. In particular, $[i,\bx,j]\in \mathfrak{D}$ satisfies $\cR(i) \cap (\cR(j) + \pi(\bx))\not=\emptyset$ if and only if $[i,\bx,j] \in B$. Moreover, $\bz\in \cR(i) \cap (\cR(j) + \pi(\bx))$ if and only if there is an infinite walk in $G_B$ starting from $[i,\bx,j]$ labeled by $(\eta^{(k)})_{k\ge 0}=(\eta_1^{(k)}\,|\, \eta_2^{(k)})_{k\ge 0}$ such that
\[
\bz = \sum_{k\ge 0} \mathbf{h}^k\pi\big(\eta_1^{(k)}\big)
\]
(see \cite[Corollary~5.9]{ST:09}). 
If $\mathcal{C}_0=\{ \cR(i) + \pi(\bx) \,:\, [\bx,i] \in H_\sigma \}
$ forms a tiling, as in the case of self-affine tiles with standard digit sets, this can be used to characterize the boundaries of the Rauzy fractal and its subtiles. This is the reason why the graph is called ``boundary graph''. This also makes the self-replicating boundary graph important for studying topological porperties of Rauzy fractals, see~\cite{ST:09}. 

We turn to the contact graph. We follow \cite[Section~3]{Thuswaldner:06} and \cite[Section~5.4]{ST:09} to define the contact graph of a unit Pisot substitution $\sigma$ over the alphabet $\cA$. Let us regard the faces of $H_\sigma$ as $(d-1)$-dimensional cubes as we did in \eqref{eq:face=cube}. If $[\bx,i],[\by,j]\in H_\sigma$ are different faces, then the intersection is either empty or a cube of dimension at most $(d-2)$. We are interested in all faces that have $(d-2)$-dimensional intersection. In particular, we need the set
\[
\mathfrak{D}_{\rm cont} = \big\{[i,\bx,j]\in\mathfrak{D}\;:\; [\mathbf{0},i]\cap[\bx,j] \text{ is a $(d-2)$-dimensional cube} \big\}.
\]

\begin{example} Figure~\ref{ContactInitial} contains the contact set for the substitutions $\sigma_1$ and $\sigma_2$.
\begin{figure}[h]
\includegraphics[width=3.5 cm]{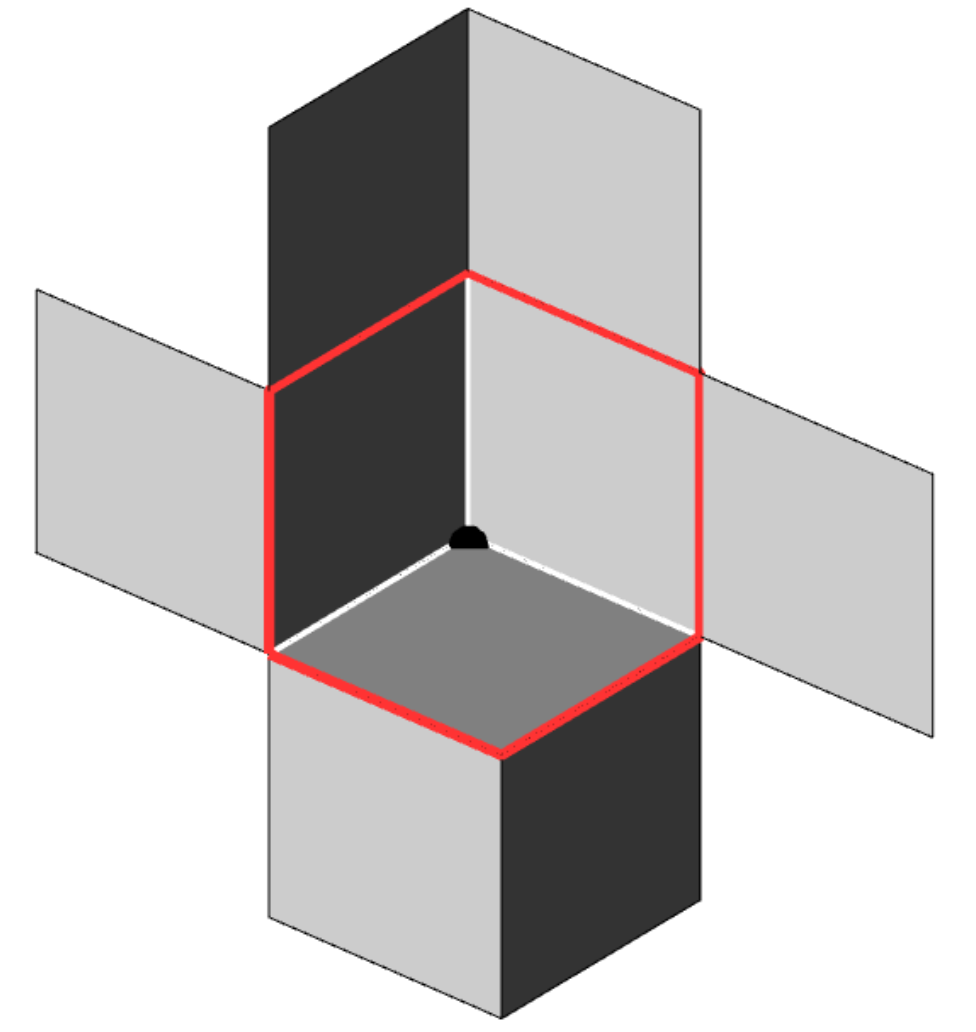} \quad\quad\quad\quad  \includegraphics[width= 3.5 cm]{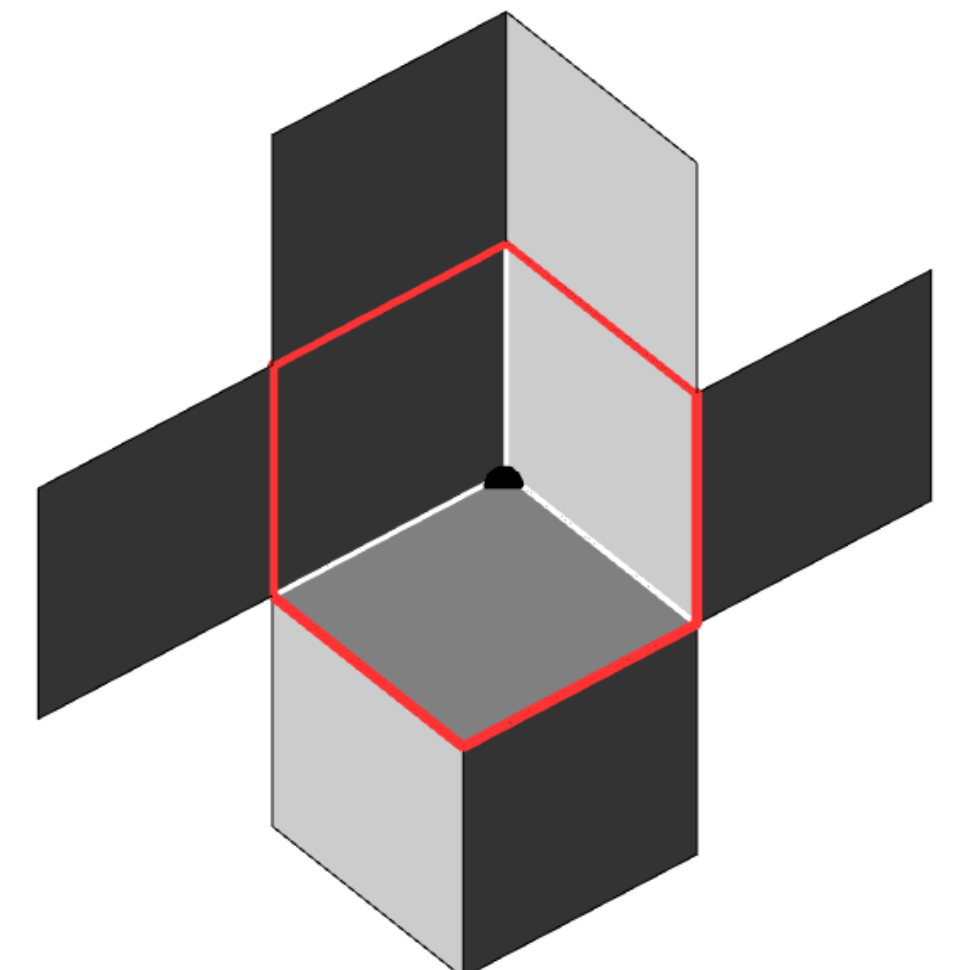}
\caption{ The white lines in the figure are $[\mathbf{0},i]\cap[\mathbf{0},j]$ and red lines are $ [\mathbf{0},i]\cap[\bx,j]$  for the unit Pisot substitution $\sigma_1$ (left) and $\sigma_2$ (right), respectively.} \label{ContactInitial}
\end{figure}
\end{example}

\begin{definition}[Self-replicating contact graph]\label{def:subsContactGraph}
The {\em (self-replicating) pre-contact graph} $G_P$ is the largest subgraph of the ambient graph for which every walk ends in an element of $\mathfrak{D}_{\rm cont}$. Then $G_C= \mathop{Red}(G_P)$ is the {\em (self-replicating) contact graph}.

We denote the set of nodes of $G_C$ by $C$. The set $C$ is called the {\em contact set}.
\end{definition}

This definition of the contact graph is equivalent to the one given in \cite[Section~5.4]{ST:09}\footnote{In \cite{ST:09} the fact that each node is contained in $\cA \times H_\sigma$ is mentioned before Propostion~5.16.} . The contact graph $G_C$ can be algorithmically computed by starting with the set of nodes  $\mathfrak{D}_{\rm cont}$. Then recursively increase this set of nodes “backwards” by adding nodes $[i,\bx,j]\in \mathfrak{D}$ that have edges to the already existing nodes. By \cite[Proposition~5.18]{ST:09}, this construction terminates after finitely many steps with the finite graph $G_P$. From this one easily obtains the contact graph $G_C$.

\subsection{A simpler variant of the graphs}
\label{sec:graph2}
We continue with the second option described in Section~\ref{sec:RauzyNeighbors}. In this case, the definitions of the according graphs are simpler but lead to larger graphs. Again we start with the ambient graph. To our knowledge, these graphs are nowhere defined explicitly, however, already in \cite{ST:09} they are implicitly used in certain proofs; see for instance~\cite[Section~5.2.1]{ST:09}.

\begin{definition}[Simple ambient graph]
\label{def:simpleambientG}
The {\em simple ambient graph} $\hat G_{\mathfrak{D}}$ is the directed graph whose nodes are the elements of $\pm(\cA\times H_\sigma)\setminus\{[i,\mathbf{0},i] \colon i\in\cA\}$.
There is an edge $[i,\bx,j]\xrightarrow{ \mathbf{l}(p_1)| \mathbf{l}(q_1)}[i',\bx',j']$ if there exist $((p_1,i,s_1),(q_1,j,t_1))\in\cP\times\cP$ with
\begin{equation*}
\begin{cases}
\sigma(i')= p_1 i s_1 \text{ and } \sigma(j')= q_1 j t_1, \\
M\bx' = \bx + \mathbf{l}(q_1) - \mathbf{l}(p_1).
\end{cases}
\end{equation*}
\end{definition}


We will often just write $[i,\bx,j]\xrightarrow{ \mathbf{l}(p_1)}[i',\bx',j']$ instead of $[i,\bx,j]\xrightarrow{ \mathbf{l}(p_1)| \mathbf{l}(q_1)}[i',\bx',j']$ because $q_1$ can be reconstructed from the remaining data.

The simple version of the self-replicating boundary graph is defined as follows. 

\begin{definition}[Simple self-replicating boundary graph]
\label{def:simpleGSsrs}
The {\em simple self-replicating boundary graph} $\hat G_B$ is the largest subgraph of the simple ambient graph for which each node belongs to a walk that ends in a loop. 
\end{definition}

Let $[\bz,j]\in H_\sigma$. In this setting we have that $\bz\in \cR(i) \cap (\cR(j) + \pi(\bx))$ if and only if there is an infinite walk in $\hat G_B$ starting from $[i,\bx,j]$ labeled by $(\mathbf{l}(p_k))_{k\ge 0}$ such that
\[
\bz = \sum_{k\ge 0} \mathbf{h}^k\pi\mathbf{l}(p_k).
\]

The simple contact graph is defined as follows. 

\begin{definition}[Simple self-replicating contact graph]\label{def:simplesubsContactGraph}
The {\em simple (self-replicating) pre-contact graph} $\hat G_P$ is the largest subgraph of the simple ambient graph for which every walk ends in an element of $\pm \mathfrak{D}_{\rm cont}$. Then $\hat G_C= \mathop{Red}(\hat G_P)$ is the {\em simple (self-replicating) contact graph}.
\end{definition}

\subsection{Relations between the two variants of graphs}\label{sec:associated}
Let $\sigma$ be a Pisot substitution. The graphs $G_{\mathfrak{D}}$, $G_B$, and $G_C$ contain the same information as the associated simple graphs $\hat G_{\mathfrak{D}}$, $\hat G_B$, and $\hat G_C$, respectively. We make this precise in the following lemma.

\begin{lemma}\label{lem:equivalence}
Let $\sigma$ be a Pisot substitution. Let $G$ be one of the graphs $G_{\mathfrak{D}}$, $G_B$, $G_C$, and let $\hat G$ be the associated simple graph. 
\begin{itemize}
\item If $[i,\bx,j]\in \mathfrak{D}$ is a node of $G$ then $[i,\bx,j]$ and  $-[i,\bx,j]$ are nodes of $\hat G$.
\item If $[i,\bx,j]$ is a node of~$\hat G$ then $[i,\bx,j] \in \mathfrak{D}$ is a node of~$G$ or $-[i,\bx,j] \in \mathfrak{D}$ is a node of~$G$.
\end{itemize}
\end{lemma}

\begin{proof}
This follows immediately from the definition of the graphs $G_{\mathfrak{D}}$, $G_B$, $G_C$ and the associated simple graphs.
\end{proof}

Thus $\hat G$ can be easily constructed from $G$ and vice versa, hence, in the sequel, w.l.o.g.\ we may work with the simple versions of the graphs.

\subsection{Connections and the $C$-corona of a graph}
Let $[i_1,\bx,i_2] \in \pm (\cA \times H_\sigma)$. Then we can view $[i_1,\bx,i_2]$ as a {\em connection} between two elements of  $\cA\times \Z^d$ in the following way. Whenever $[\by,i_1], [\by+\bx,i_2] \in \cA\times \Z^d$ we say that $[i_1,\bx,i_2]$ {\em connects} $[\by,i_1]$ and $[\by+\bx,i_2]$ (this relation is symmetric, so $[i_1,\bx,i_2]$ also connects $[\by+\bx,i_2]$ and $[\by,i_1]$).  

Let $m,m_1,m_2 \in \pm  (\cA \times H_\sigma)$ be given. 
If there exist $[\bx_1,i_1],[\bx_2,i_2],[\bx_3,i_3]\in H_\sigma$ such that
$m_1$ connects $[\bx_1,i_1]$ and $[\bx_2,i_2]$, $m_2$ connects $[\bx_2,i_2]$ and $[\bx_3,i_3]$, and $m$ connects $[\bx_1,i_1]$ and $[\bx_3,i_3]$ we say that the connection $m$ can be {\em decomposed} into the connections $m_1$ and $m_2$. 

We are particularly interested in chains of connections formed by elements of $\pm C$. Indeed, let $[\bx_1,i_1],[\bx_2,i_2] \in H_\sigma$. If there is $r\in \N$ and $[\by_k,j_{k}]\in H_\sigma$, $k\in\{0,\ldots, r\}$, with 
$[\by_0,j_0] = [\bx_1,i_1]$, $[\by_r,j_r] = [\bx_2,i_2]$,  and 
$[j_{k},\by_{k+1}-\by_k,j_{k+1}] \in \pm C \cup\{[i,\mathbf{0},i]\colon i\in\cA\}$ for $k\in\{0,\ldots, r-1\}$, we say that the two elements $[\bx_1,i_1],[\bx_2,i_2] \in H_\sigma$ can be connected by $r$ subsequent {\em $\pm C$-connections} and write $[\bx_1,i_1] \overset{r}{\sim} [\bx_2,i_2]$. If  $[\bx_1,i_1],[\bx_2,i_2] \in -H_\sigma$, we write $[\bx_1,i_1] \overset{r}{\sim} [\bx_2,i_2]$ if $[-\bx_1,i_1] \overset{r}{\sim} [-\bx_2,i_2]$


We need the following definition of {\em $C$-corona}  for subgraphs of the simple ambient graph.

\begin{definition}[$C$-corona]\label{def:corona}
Let $\hat G$ be a subgraph of the simple ambient graph.
The {\em $C$-corona} $C$-Corona$(\hat G)$ is the subgraph of the ambient graph whose nodes are given by 
\[
\big\{[i_1,\bx, i_2] \in \pm(\cA\times H_\sigma) \setminus\{[i,\mathbf{0},i] \colon i\in\cA\} \;:\;  \text{there is } [i_1,\by,j]\in \hat G \text{ with } 
[\by,j] \overset{1}{\sim} [\bx, i_2] 
 \big\}.
\] 
\end{definition}

In other words, each node of $C$-Corona$(\hat G)$ can be reached from a node of $\hat G$ by a $\pm C$-connection. Because $C$ is a finite set, it is easy to construct $C$-Corona$(\hat G)$ from $\hat G$. 

\section{A neighbor finding algorithm for Rauzy fractals}\label{sec:RALGOR}

\subsection{Formulation of the new neighbor finding algorithm}
Our goal is to show that  by forming $C$-coronas and reducing iteratively, we can obtain $G_B$ from $G_C$ by Algorithm~\ref{algRauzyneighbor}.
\begin{algorithm}[H]\
  \caption{(Construction of the self-replicating neighbor graph of a Pisot substitution $\sigma$.)}
   \label{algRauzyneighbor}
  \begin{algorithmic}
    \REQUIRE Contact graph $G_C$
    \ENSURE Neighbor graph $G_B$
    \STATE $p\assign 1$ 
    \STATE $A[1]\assign \hat G_C$
    \REPEAT 
      \STATE $p\assign p+1$
      \STATE $A[p] \assign\mathop{Red}$( $C$-Corona$(A[p-1]) )$
    \UNTIL{$A[p] = A[p-1]$}
    
    $\hat G_B \assign A[p]$
  \end{algorithmic}
 \end{algorithm}
Recall that we saw in Section~\ref{sec:associated}, $G_C$ and $G_B$ are equivalent to $\hat G_C$ and $\hat G_B$, respectively.

\subsection{A proof for the new algorithm}
In this section we prove that Algorithm~\ref{algRauzyneighbor} terminates and calculates the self-replicating boundary graph of a given Pisot substitution. Although the proof is inspired by the according proof in the setting of self-affine tiles with standard digit set, it is more involved.

In the setting of self-affine tiles with standard digit set  we used that the contact set $R$ contains a basis of the lattice $\Z^d$. The analog of this observation is contained in the following auxiliary result. It shows that we can jump from one element of $H_\sigma$ to another by finitely many subsequent $\pm C$-connections.

\begin{lemma}\label{lem:Rsteps}
Let $[\bx_1,i_1],[\bx_2,i_2] \in  H_\sigma$. Then there is $r\in \N$ with  $[\bx_1,i_1] \overset{r}{\sim} [\bx_2,i_2]$.
\end{lemma}

\begin{proof}
The proof draws on results from \cite{Thuswaldner:06}.\footnote{In \cite{Thuswaldner:06}, suffixes are used instead of prefixes, however, according to \cite[Section~4.3]{Thuswaldner:06} this is immaterial.} 
For $n\in\N$, let $\cR_n(i)$, $i\in \cA$, be given as in \eqref{eq:rapprox}. Translations of these sets form the tiling $\mathcal{I}_n$ defined in Lemma~\ref{lem:nthtiling}. 
%
By definition, the elements of $\mathcal{I}_n$ are finite unions of $(d-1)$-dimensional parallelotopes. Two distinct elements of $\mathcal{I}_n$ intersect in a (possibly empty) finite union of  parallelotopes whose (topological) dimension is less than or equal to $d-2$.

By \cite[Lemma~4.2]{Thuswaldner:06} there is $n\in\N$ such that following holds. If $\cR_n(i),\cR_n(j)+\pi(\bx)\in \mathcal{I}_n$ share a $(d-2)$-dimensional parallelotope then $[i,\bx,j] \in C$. 
From the definition of $H_\sigma$ we see that $\cR_n(i) + \pi(\bx), \cR_n(j)+\pi(\by) \in \mathcal{I}_n$ implies that either $\cR_n(i), \cR_n(j)+\pi(\by-\bx) \in \mathcal{I}_n$ or $\cR_n(i) + \pi(\bx-\by), \cR_n(j) \in \mathcal{I}_n$. Thus, if two distinct elements $\cR_n(i) + \pi(\bx), \cR_n(j)+\pi(\by) \in \mathcal{I}_n$ have $(d-2)$-dimensional intersection, then $[i,\by-\bx,j] \in \pm C$.

Now choose $[\bx_1,i_1],[\bx_2,i_2]\in H_\sigma$ arbitrary. Set $\mathcal{J}_0=\{[\bx_1,i_1]\}$ and for $k\ge 1$ inductively define
\[
\mathcal{J}_k= \big\{[\by,j] \in H_\sigma \colon (\cR_n(j)+\pi(\by)) \cap X 
\text{ is at least $(d-2)$-dimensional for some $X\in \mathcal{J}_{k-1}$}
\big\}.
\]
Note that $\mathcal{J}_k$ is contained in the subset of all elements of $H_\sigma$ that can be connected to $[\bx_1,i_1]$ by at most $k$ $\pm C$-connections. Thus, setting $\mathcal{J}=\bigcup_{k\ge 0}\mathcal{J}_k$,  it remains to show that $[\bx_2,i_2] \in \mathcal{J}$. Assume on the contrary that this is not the case. Then the sets
\[
A=\bigcup_{[\bx,i]\in \mathcal{J}} (\cR_n(i) + \pi(\bx)) \quad\text{and}\quad
B=\bigcup_{[\bx,i]\in H_\sigma \setminus \mathcal{J}} (\cR_n(i) + \pi(\bx))
\]
are both $(d-1)$-dimensional,  $A \cup B = \bv^\bot\simeq \R^{d-1}$, and the dimension of $A\cap B$ is at most $d-3$. Thus $A\cap B$ is a $(d-3)$-dimensional cut of a $(d-1)$-dimensional space. This contradicts a classical theorem of Mazurkiewicz; see~\cite[\S59, II, Theorem~1]{Kuratowski:68}.
\end{proof}

Let $[i,\bm,j] \in \pm B$ be an element of the neighbor set. We say that $[i,\bm,j]$ has {\em contact degree $q$} if $q$ is minimal with $[\mathbf{0},i]\overset{q}{\sim}[\bm,j]$. Because the nonnegative integer $q$ is a well-defined positive integer by Lemma~\ref{lem:Rsteps}, 
we may set $\mathrm{cdeg}([i,\bm,j]):=q$. Since $B$ is a finite set, we gain the following result.

\begin{lemma}\label{lem:Bfinitecdeg}
There is $p\in \mathbb{N}$ such that each $[i,\bm,j] \in B$ satisfies $\mathrm{cdeg}([i,\bm,j]) \le p$.  
\end{lemma}

Let $w$ be a walk in $\hat G_B$. Then $\mathrm{cdeg}(w)$ denotes the maximum of the degree of the nodes of $w$. We have $\mathrm{cdeg}(w) \le p$. 

We need the following lemma, which shows how the contact degree behaves under the dual substitution. It is an analog of Lemma~\ref{lem:mono} from the setting of self-affine tiles.

\begin{lemma}\label{lem:contactdegree}
Let $[\bx_1,i_1], [\bx_2,i_2],[\by_1,j_1], [\by_2,j_2] \in H_\sigma$ be given in a way that $[\by_1,j_1] \in \sigma^*[\bx_1,i_1]$ and $[\by_2,j_2] \in \sigma^*[\bx_2,i_2]$. Then
$[\by_1,j_1] \overset{q}{\sim} [\by_2,j_2]$ implies $[\bx_1,i_1] \overset{q}{\sim} [\bx_2,i_2]$.
\end{lemma}

\begin{proof}
It suffices to prove the lemma for $q=1$. The general case immediately follows from this by induction.
Let $q=1$ and $[\bx_1,i_1], [\bx_2,i_2],[\by_1,j_1], [\by_2,j_2]$ be as in the statement of the lemma.
 
By definition, the relation $[\by_1,j_1] \overset{1}{\sim} [\by_2,j_2]$ implies that $[j_1,\by_2-\by_1,j_2] \in \pm C\cup\{[i,\mathbf{0},i]; i\in \cA\}$. 
Because $[\by_1,j_1] \in \sigma^*[\bx_1,i_1]$ and $[\by_2,j_2] \in \sigma^*[\bx_2,i_2]$ There exist $p_1,p_2,s_1,s_2\in\cA^*$ such that
\begin{equation}\label{eq:twopredecessors}
\begin{split}
[\by_1,j_1] = \big[ M^{-1}(\bx_1+\mathbf{l}(p_1)),j_1 \big] \quad\text{with }\sigma(j_1)=p_1i_1s_1, \\
[\by_2,j_2] = \big[ M^{-1}(\bx_2+\mathbf{l}(p_2)),j_2 \big] \quad\text{with }\sigma(j_2)=p_2i_2s_2.
\end{split}
\end{equation}
By the definition the simple ambient graph, \eqref{eq:twopredecessors} implies that
\[
[i_1,\bx_2-\bx_1,i_2] \xrightarrow{\mathbf{l}(p_1)} [j_1,\by_2-\by_1,j_2]
\]
is an edge in the simple ambient graph. Thus, because $[j_1,\by_2-\by_1,j_2]\in \pm C$, by the definition of the simple contact graph $[i_1,\bx_2-\bx_1,i_2]\in \pm C$ and, hence, $[\bx_1,i_1]\overset{1}{\sim}[\bx_2,
i_2]$. 
\end{proof}

The following result is a version of Lemma~\ref{lem:2} in the setting of Pisot substitutions. For two finite words $w_1,w_2\in\cA^*$, we write $w_2\prec w_1$ if $w_2$ is a prefix of $w_1$ ($w_1=w_2w_3$ for some $w_3\in\cA^*$).

\begin{lemma}\label{lem:type1characterization}
The following assertions are equivalent.
\begin{itemize}
\item In $\hat G_B$ there exists a walk 
\begin{equation*}
w:\; [i_0,\bm_0,j_0] \xrightarrow{\mathbf{l}(p_0)|\mathbf{l}(q_0)} [i_1,\bm_1,j_1] \xrightarrow{\mathbf{l}(p_1)|\mathbf{l}(q_1)} \cdots \xrightarrow{\mathbf{l}(p_{k-1})|\mathbf{l}(q_{k-1})} [i_k,\bm_k,j_k] 
\end{equation*}
with $[\bm_0,j_0]\in H_\sigma$ and $\mathrm{cdeg}(w)\le q$.
\item There exist sequences $([\ba_\ell, i_\ell])_{0\le\ell\le k}$ and $([\bb_\ell, j_\ell])_{0\le\ell\le k}$ in $H_\sigma$, where $[\ba_0,i_0]=[\mathbf{0},i_0]$, $[\bb_0,j_0]=[\bm_0,j_0]$, and
\begin{equation}\label{eq:abfaces}
\begin{split}
[\ba_{\ell}, i_{\ell}]&=\big[M^{-1}(\ba_{\ell-1}+\mathbf{l}(p_{\ell-1})), i_{\ell} \big] \in (\sigma^*)^\ell[\mathbf{0},i_0]
\quad\text{and}\\
[\bb_\ell, j_\ell]&=\big[M^{-1}(\bb_{\ell-1}+\mathbf{l}(q_{\ell-1})), j_{\ell} \big] \in (\sigma^*)^\ell[\bm_0,j_0]
\end{split}
\end{equation}
with $\bm_\ell=\bb_\ell-\ba_\ell$, $p_{\ell-1}i_{\ell-1}\prec \sigma(i_\ell)$ and $q_{\ell-1}j_{\ell-1}\prec \sigma(j_\ell)$, and such that $[\ba_\ell, i_\ell]\overset{q}{\sim}[\bb_\ell, j_\ell]$ ($\ell\in\{1,\ldots, k\}$). 
\end{itemize}
\end{lemma}

\begin{proof}
This follows easily by induction on $k$ and by using the definition of the edges of $\hat G_B$
\end{proof}

This will enable us to prove an analog of Proposition~\ref{lem:2} in the setting of Rauzy fractals.

\begin{proposition}\label{lem:Rauzy2}
Let $w:\; m_0 \xrightarrow{\eta_0} m_1 \xrightarrow{\eta_1}m_2 \xrightarrow{\eta_2} \cdots$ be a walk in $\hat G_B$ with $\mathrm{cdeg}(w)>1$. Then there exist walks
\[
w_1:\;m_0^{(1)}\xrightarrow{\eta_0^{(1)}} m_1^{(1)} \xrightarrow{\eta_1^{(1)}} m_2^{(1)} \xrightarrow{\eta_2^{(1)}} \cdots
\quad \hbox{and} \quad
 w_2:\;m_0^{(2)}\xrightarrow{\eta_0^{(2)}} m_1^{(2)} \xrightarrow{\eta_1^{(2)}} m_2^{(2)} \xrightarrow{\eta_2^{(2)}} \cdots
\]
in $\hat G_B$,
satisfying $\mathrm{cdeg}(w_1)< \mathrm{cdeg}(w)$, and $\mathrm{cdeg}(w_2) =1$, such that 
the connection $m_k$ can be decomposed into the connections $m_k^{(1)}$ and $m_k^{(2)}$ for each $k\in\N$.
\end{proposition}

\begin{proof}
Fix $k\in\N$ and let $w|_k$ be the walk consisting of the first $k$ edges of $w$ and
write out the nodes and labels of the walk $w|_k$ by 
\[
w|_k:\; [i_0,\bm_0,j_0] \xrightarrow{\mathbf{l}(p_0)|\mathbf{l}(q_0)} [i_1,\bm_1,j_1] \xrightarrow{\mathbf{l}(p_1)|\mathbf{l}(q_1)} \cdots \xrightarrow{\mathbf{l}(p_{k-1})|\mathbf{l}(q_{k-1})} [i_k,\bm_k,j_k].
\]
W.l.o.g.\ we assume that $[\bm_0,j_0] \in H_\sigma$ (otherwise we just consider $[-\bm_0,i_0]$ instead). Lemma~\ref{lem:type1characterization} implies that for each $\ell\in\{0,\ldots,k\}$ there exist faces $[\ba_\ell,i_\ell] \in (\sigma^\ell)^*[\mathbf{0},i_0]$ and $[\bb_\ell,j_\ell]\in(\sigma^\ell)^*[\bm_0,j_0]$ as defined in \eqref{eq:abfaces} that satisfy  $[\ba_\ell,i_\ell] \overset{q}{\sim} [\bb_\ell,j_\ell]$ with $q\geq \mathrm{cdeg}(w)\geq 2$. Thus by Lemma~\ref{lem:type1characterization} the walk $w|_k$ can be written as 
\[
w|_k:\; [i_0,\bb_0-\ba_0,j_0] \xrightarrow{\mathbf{l}(p_0)|\mathbf{l}(q_0)} [i_1,\bb_1-\ba_1,j_1] \xrightarrow{\mathbf{l}(p_1)|\mathbf{l}(q_1)} \cdots \xrightarrow{\mathbf{l}(p_{k-1})|\mathbf{l}(q_{k-1})} [i_k,\bb_k-\ba_k,j_k].
\]
with $[\ba_k,i_k] \overset{q}{\sim} [\bb_k,j_k]$. Moreover, there exists $[\by_k,h_k] \in H_\sigma$ such that $[\ba_k,i_k] \overset{q-1}{\sim} [\by_k,h_k]\overset{1}{\sim} [\bb_k,j_k]$. By Lemma~\ref{lem:nthtiling} there exists $[\by_\ell,h_\ell]$ such that $[\by_k,h_k]\in(\sigma^{k-\ell})^*[\by_\ell,h_\ell]$. By the definition of $\sigma^*$ this implies that
\begin{equation}\label{eq:yfaces}
[\by_\ell, h_\ell]=\big[ M^{-1}(\by_{\ell-1}+ t_{\ell-1}), h_\ell \big] \in (\sigma^*)^\ell[\by_0,h_0],
\end{equation}
where $t_{\ell-1}h_{\ell-1} \prec \sigma(h_{\ell})$ for $\ell\in\{1,\ldots,k\}$. Moreover, by Lemma~\ref{lem:contactdegree} for all $\ell\in\{0,\ldots,k\}$ we have $[\ba_\ell,i_\ell] \overset{q-1}{\sim} [\by_\ell,h_\ell]\overset{1}{\sim} [\bb_\ell,j_\ell]$. Thus, applying Lemma~\ref{lem:type1characterization} to $([\ba_\ell,i_\ell])_{0\le\ell\le k}$ and $([\by_\ell,h_\ell])_{0\le\ell\le k}$ we see that there exists a walk
\[
w_{1,k}:\; [i_0,\by_0-\ba_0,h_0] \xrightarrow{\mathbf{l}(p_0)|\mathbf{l}(t_0)}  [i_1,\by_1-\ba_1,h_1] \xrightarrow{\mathbf{l}(p_1)|\mathbf{l}(t_1)} \cdots \xrightarrow{\mathbf{l}(p_{k-1})|\mathbf{l}(t_{k-1})}  [i_k,\by_k-\ba_k,h_k]
\]
with $\mathrm{cdeg}(w_{1,k})\le q-1$. We know that $[\bb_0-\by_0,j_0] \in \pm H_\sigma$. Assume that we have $[\bb_0-\by_0,j_0] \in H_\sigma$ (the other case can be treated analogously). Then, shifting by $-M^{-\ell}\by_0$ and using the definition of $\sigma^*$ we gain
\[
[\bb_\ell - M^{-\ell}\by_0, j_\ell] \in (\sigma^*)^\ell[\bb_0-\by_0,j_0] \quad\hbox{and} \quad
[\by_\ell - M^{-\ell}\by_0, h_\ell] \in (\sigma^*)^\ell[\mathbf{0},h_0].
\]
Thus we may apply  Lemma~\ref{lem:type1characterization} to $([\by_\ell - M^{-\ell}\by_0, h_\ell])_{0\le\ell\le k}$ and $([\bb_\ell - M^{-\ell}\by_0, j_\ell])_{0\le\ell\le k}$ to exhibit a walk
\[
w_{2,k}:\; [h_0,\bb_0-\by_0,j_0] \xrightarrow{\mathbf{l}(t_0)|\mathbf{l}(q_0)}  [h_1,\bb_1-\by_1,j_1] \xrightarrow{\mathbf{l}(t_1)|\mathbf{l}(q_1)} \cdots \xrightarrow{\mathbf{l}(t_{k-1})|\mathbf{l}(q_{k-1})}  [h_k,\bb_k-\by_k,j_k]
\]
with $\mathrm{cdeg}(w_{2,k})= 1$. For $\ell\in\{0,\ldots, k\}$ consider the nodes $[i_\ell,\bb_\ell-\ba_\ell,j_\ell]$, $[i_\ell,\by_\ell-\ba_\ell,h_\ell]$, and $[h_\ell,\bb_\ell-\by_\ell,j_\ell]$.  Then $[i_\ell,\bb_\ell-\ba_\ell,j_\ell]$ connects $[\ba_\ell,i_\ell]$ and $[\bb_\ell,j_\ell]$. Since $[i_\ell,\by_\ell-\ba_\ell,h_\ell]$ connects $[\ba_\ell,i_\ell]$ and $[\by_\ell,h_\ell] \in H_\sigma$ and $[h_\ell,\bb_\ell-\by_\ell,j_\ell]$ connects $[\by_\ell,h_\ell]$ and $[\bb_\ell,j_\ell]$, we conclude that $[i_\ell,\bb_\ell-\ba_\ell,j_\ell]$ can be decomposed into $[i_\ell,\by_\ell-\ba_\ell,h_\ell]$ and $[h_\ell,\bb_\ell-\by_\ell,j_\ell]$. Thus we proved the lemma for walks of length $k$.

However, since a given node of $w$ can have only finitely many decompositions in two connections of contact degree $\mathrm{cdeg}(w)-1$ and $1$, respectively, by a Cantor diagonal argument we may choose $w_1$ and $w_2$ as asserted in the lemma for all (infinitely many) edges. This finishes the proof.
\end{proof}

We are now in a position to prove that Algorithm~\ref{algRauzyneighbor} yields the neighbor graph after finitely many steps.

\begin{theorem}
Let $\sigma$ be a Pisot substitution. Then Algorithm~\ref{algRauzyneighbor} terminates after finitely many steps and has the self-replicating neighbor graph $G_B$ of $\sigma$ as its output.
\end{theorem} 

\begin{proof}
 Let $A[1] = \hat G_C$ and $A[p] = \mathop{Red}(C\text{-Corona}(A[p-1]))$ for $p>1$. We prove by induction that $A[p]$ contains all infinite walks $w$ of $\hat G_B$ having $\mathop{cdeg}(w) \le p$. 

Let $w$ be an infinite walk in $\hat G_B$ with $\mathop{cdeg}(w) = 1$ then, by definition, $w$ is a walk whose nodes are contained in $\pm C$. Thus $w$ is a walk in $\hat G_C$ and, hence, in $A[1]$. This constitutes the induction start.

To prove the induction step we assume that $A[p-1]$ consists of all infinite walks of $\hat G_B$ having contact degree less than or equal to $p-1$. Let $w:\; m_0 \xrightarrow{\eta_0} m_1 \xrightarrow{\eta_1}m_2 \xrightarrow{\eta_2} \cdots$ be an infinite walk in $\hat G_B$ having $\mathop{cdeg}(w) \le p$. Then, by Proposition~\ref{lem:Rauzy2} and by the induction hypothesis we know that there exist infinite walks $w_1:\;m_0^{(1)}\xrightarrow{\eta_0^{(1)}} m_1^{(1)} \xrightarrow{\eta_1^{(1)}} m_2^{(1)} \xrightarrow{\eta_2^{(1)}} \cdots
\in A[p-1]$ and $w_2:\;m_0^{(2)}\xrightarrow{\eta_0^{(2)}} m_1^{(2)} \xrightarrow{\eta_1^{(2)}} m_2^{(2)} \xrightarrow{\eta_2^{(2)}} \cdots \in A[1]$ such that $m_k=[i_k,\bm_k,j_k]$ can be decomposed into the connections $m_k^{(1)}=[i_k,\bm_k^{(1)},h_k]$ and $m_k^{(2)}$ for each $k\in \N$.  
Since $m_k^{(2)} \in \pm C$, by possibly replacing $m_k^{(1)}$ by its negative (which is still an element of $A[p-1]$), we have 
$[\bm_k,j_k] \overset{1}{\sim}[\bm_k^{(1)},h_k]$
 and, hence, $m_k=[i_k,\bm_k,j_k] \in (C\text{-Corona}(A[p-1])$. Because $k$ was arbitrary, $w$ is an infinite walk in $C$-Corona$(A[p-1])$ and, therefore, an infinite walk in $A[p]=\mathop{Red}(C\text{-Corona}(A[p-1]))$. 

Since by Lemma~\ref{lem:Bfinitecdeg} there is $q\in\N$ such that each walk $w$ in $\hat G_B$ satisfies $\mathop{cdeg}(w) \le q$, we conclude that $A[q+1]=A[q]$. Thus the algorithm terminates after finitely many steps and returns $\hat G_B$, and, hence, $G_B$, as desired.
\end{proof}

\section{Examples}\label{sec:Examples}
We consider an  example and explain how our algorithm works. We set $\mathbf{e}_1=(1,0,0),\mathbf{e}_2=(0,1,0), \mathbf{e}_3=(0,0,1)$  as the basis of $\R^3$. 
\begin{example}
Recall that 
$$\sigma_1: \left\{\begin{aligned}
1&\longrightarrow 1112\\
2&\longrightarrow 113\\
3&\longrightarrow 1\\
\end{aligned}
\right.
\quad \quad\quad\quad
\sigma_2: \left\{\begin{aligned}
1&\longrightarrow 112\\
2&\longrightarrow 1113\\
3&\longrightarrow 1\\
\end{aligned}
\right.
$$
We start by computing the contact graph for $\sigma_1$ and $\sigma_2$. By  construction, we first get the set $\mathfrak{D}_{\rm cont}$, which includes all faces having $(d-2)$-dimensional intersection. For $\sigma_1$, 
\[
\mathfrak{D}_{\rm cont}=\{[1,\mathbf{e}_2,1],[1,\mathbf{e}_3,1],[1,\mathbf{0},2],[1,\mathbf{0},3],[2,\mathbf{e}_3,1],[2,\mathbf{e}_1-e_2,1],[2,\mathbf{0},3],[3,\mathbf{e}_1-\mathbf{e}_3,1],[3,\mathbf{e}_2-\mathbf{e}_3,2]\},
\]
and for $\sigma_2$
\[
\mathfrak{D}_{\rm cont}=\{[1,\mathbf{e}_2,2],[1,\mathbf{e}_3,1],[1,\mathbf{0},2],[1,\mathbf{0},3],[2,\mathbf{e}_3,1],[2,\mathbf{e}_1,2],[2,\mathbf{0},3],[3,\mathbf{e}_1-\mathbf{e}_3,1],[3,\mathbf{e}_2-\mathbf{e}_3,2]\}
\]
Then, by Definition \ref{def:simplesubsContactGraph}, we obtain the simple self-replicating contact graph. Now, applying Algorithm \ref{algRauzyneighbor} to the simple self-replicating contact graph, we can compute the simple self-replicating neighbor graph.

Figure \ref{Nei_sigma1} presents the  self-replicating contact  graph for
 $\sigma_1$. It has $14$ vertices, while the simple contact neighbor graph has $26$ vertices. We recall that we exclude $[i,\mathbf{0},j]$ for $i\geq j$. The dashed lines in the figure show edges of type $2$, according to Definition \ref{def:ambientG}. Considering the graph as a subgraph of the simple self-replicating contact graph, then the edge $[i,x,j]\rightarrow[i',x',j']$ of type $2$ would change to an edge $[i,x,j]\rightarrow[j',-x',i']$. By Algorithm \ref{algRauzyneighbor}, it turns out that the (simple) self-replicating neighbor graph is exactly the same as the (simple) self-replicating contact graph.

\begin{figure}[h]
\includegraphics[width=13 cm]{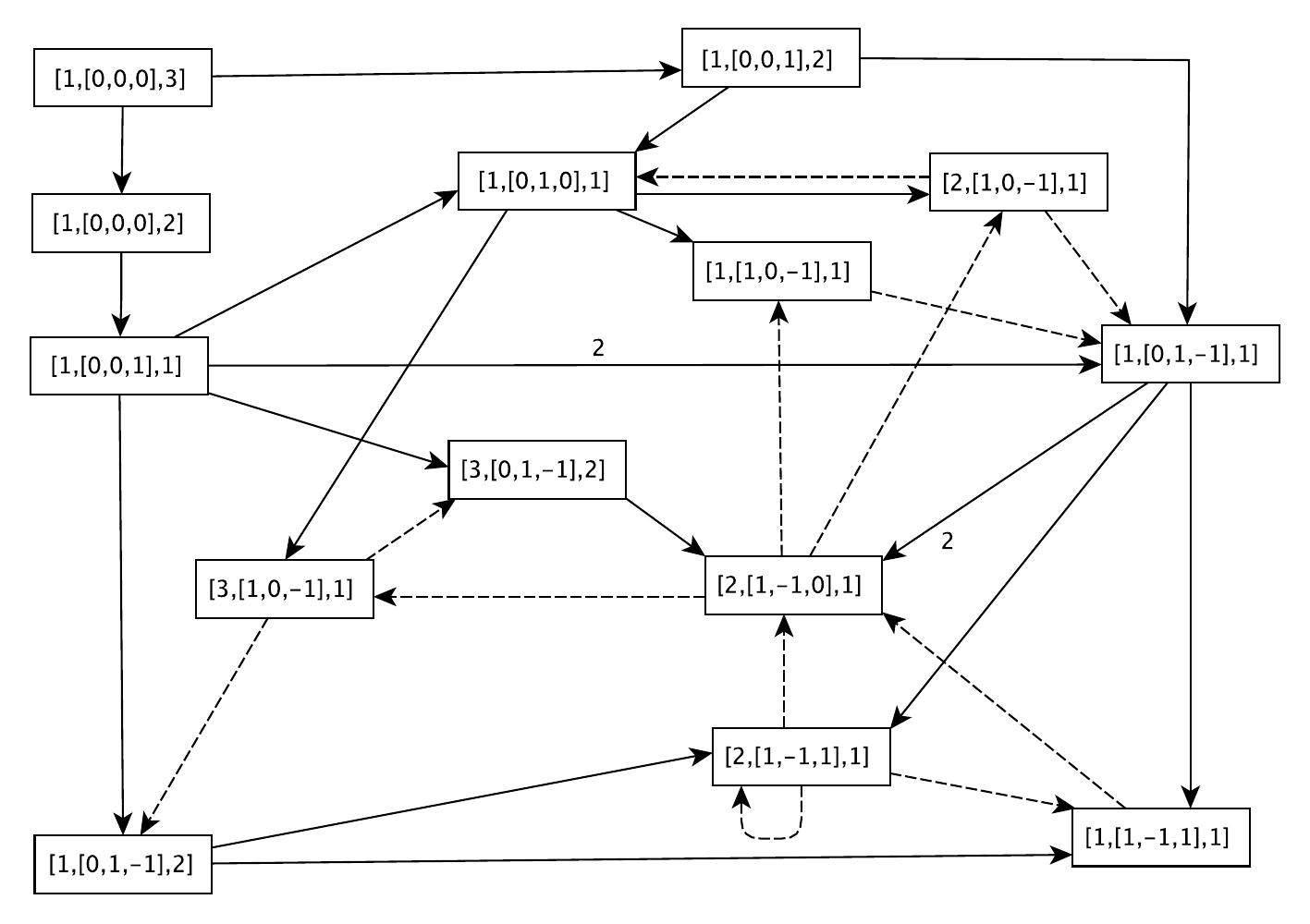} 
\caption{Self-replicating contact graph and self-replicating neighbor graph of $\sigma_1$. The number ``2'' on some of the arrows indicates that there are actually two edges between the associated vertices.} \label{Nei_sigma1}
\end{figure}

In Figure \ref{Nei_sigma2} we depicted the self-replicating neighbor graph of $\sigma_2$. It was obtained by applying our new algorithm to the self-replicating contact graph, whose 15 vertices are the dark grey nodes in the same figure.

\begin{figure}[h]
\includegraphics[width=13 cm]{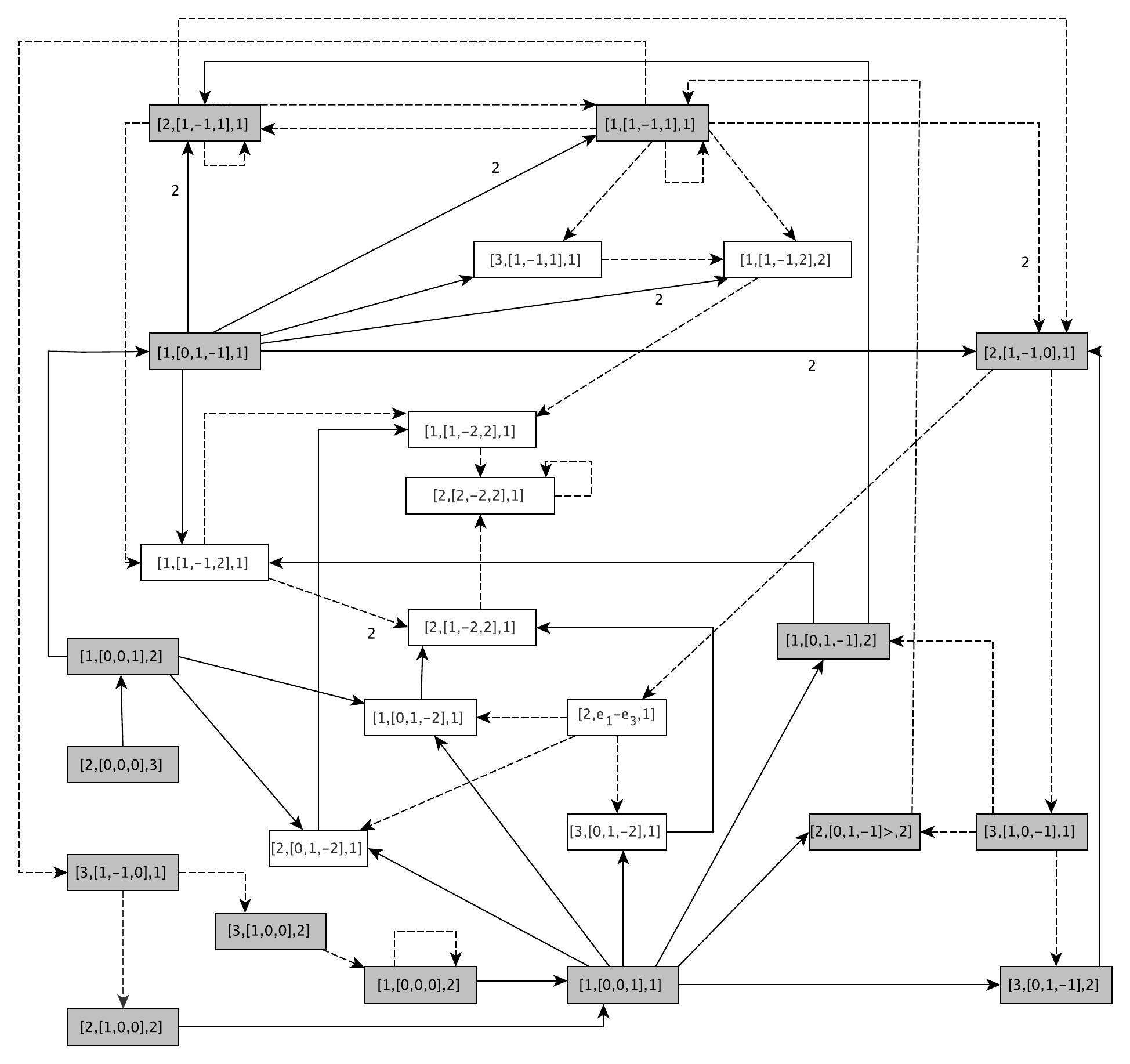} 
\caption{Self-replicating contact  graph (dark grey nodes) and  self-replicating neighbor graph of $\sigma_2$. The number ``2'' on some of the arrows indicates that there are actually two edges between the associated vertices.}\label{Nei_sigma2}
\end{figure}

\end{example}

\bibliographystyle{abbrv}
\bibliography{RauzyNeighbors}

\def\cprime{$'$} \def\cprime{$'$} \def\cprime{$'$}
\begin{thebibliography}{10}

\bibitem{Akiyama:02}
S.~Akiyama.
\newblock On the boundary of self affine tilings generated by {P}isot numbers.
\newblock {\em J. Math. Soc. Japan}, 54(2):283--308, 2002.

\bibitem{ABBLS:15}
S.~Akiyama, M.~Barge, V.~Berth\'{e}, J.-Y. Lee, and A.~Siegel.
\newblock On the {P}isot substitution conjecture.
\newblock In {\em Mathematics of aperiodic order}, volume 309 of {\em Progr.
  Math.}, pages 33--72. Birkh\"{a}user/Springer, Basel, 2015.

\bibitem{AT:05}
S.~Akiyama and J.~M. Thuswaldner.
\newblock The topological structure of fractal tilings generated by quadratic
  number systems.
\newblock {\em Comput. Math. Appl.}, 49(9-10):1439--1485, 2005.

\bibitem{AL:19}
L.-X. An and K.-S. Lau.
\newblock Characterization of a class of planar self-affine tile digit sets.
\newblock {\em Trans. Amer. Math. Soc.}, 371(11):7627--7650, 2019.

\bibitem{ABFJ:07}
P.~Arnoux, V.~Berth\'{e}, T.~Fernique, and D.~Jamet.
\newblock Functional stepped surfaces, flips, and generalized substitutions.
\newblock {\em Theoret. Comput. Sci.}, 380(3):251--265, 2007.

\bibitem{Arnoux-Ito:01}
P.~Arnoux and S.~Ito.
\newblock Pisot substitutions and {R}auzy fractals.
\newblock {\em Bull. Belg. Math. Soc. Simon Stevin}, 8(2):181--207, 2001.
\newblock Journ{\'e}es Montoises d'Informatique Th{\'e}orique
  (Marne-la-Vall{\'e}e, 2000).

\bibitem{BG:13}
M.~Baake and U.~Grimm.
\newblock {\em Aperiodic order. {V}ol. 1}, volume 149 of {\em Encyclopedia of
  Mathematics and its Applications}.
\newblock Cambridge University Press, Cambridge, 2013.
\newblock A mathematical invitation, With a foreword by Roger Penrose.

\bibitem{BG:20}
M.~Baake and U.~Grimm.
\newblock Fourier transform of {R}auzy fractals and point spectrum of 1{D}
  {P}isot inflation tilings.
\newblock {\em Doc. Math.}, 25:2303--2337, 2020.

\bibitem{Bandt91}
C.~Bandt.
\newblock Self-similar sets 5. {I}nteger matrices and fractal tilings of
  $\mathbb{R}^n$.
\newblock {\em Proc. Amer. Math. Soc.}, 112(2):549--562, 1991.

\bibitem{Barge:16}
M.~Barge.
\newblock Pure discrete spectrum for a class of one-dimensional substitution
  tiling systems.
\newblock {\em Discrete Contin. Dyn. Syst.}, 36(3):1159--1173, 2016.

\bibitem{Barge:18}
M.~Barge.
\newblock The {P}isot conjecture for {$\beta$}-substitutions.
\newblock {\em Ergodic Theory Dynam. Systems}, 38(2):444--472, 2018.

\bibitem{BD:02}
M.~Barge and B.~Diamond.
\newblock Coincidence for substitutions of {P}isot type.
\newblock {\em Bull. Soc. Math. France}, 130(4):619--626, 2002.

\bibitem{BS:05}
V.~Berth\'{e} and A.~Siegel.
\newblock Tilings associated with beta-numeration and substitutions.
\newblock {\em Integers}, 5(3):A2, 46, 2005.

\bibitem{CANTBST}
V.~Berth{\'e}, A.~Siegel, and J.~M. Thuswaldner.
\newblock {Substitutions, Rauzy fractals, and tilings}.
\newblock In {\em Combinatorics, Automata and Number Theory}, volume 135 of
  {\em Encyclopedia of Mathematics and its Applications}. Cambridge University
  Press, 2010.

\bibitem{BST:19}
V.~Berth\'{e}, W.~Steiner, and J.~M. Thuswaldner.
\newblock Geometry, dynamics, and arithmetic of {$S$}-adic shifts.
\newblock {\em Ann. Inst. Fourier (Grenoble)}, 69(3):1347--1409, 2019.

\bibitem{Berthe-Vuillon:00}
V.~Berth\'{e} and L.~Vuillon.
\newblock Tilings and rotations on the torus: a two-dimensional generalization
  of {S}turmian sequences.
\newblock {\em Discrete Math.}, 223(1-3):27--53, 2000.

\bibitem{Canterini-Siegel:01a}
V.~Canterini and A.~Siegel.
\newblock Automate des pr\'efixes-suffixes associ\'e \`a une substitution
  primitive.
\newblock {\em J. Th\'eor. Nombres Bordeaux}, 13(2):353--369, 2001.

\bibitem{CT:16}
G.~R. Conner and J.~M. Thuswaldner.
\newblock Self-affine manifolds.
\newblock {\em Adv. Math.}, 289:725--783, 2016.

\bibitem{DLN:22}
G.~Deng, C.~Liu, and S.-M. Ngai.
\newblock A class of self-affine tiles in {$\Bbb R^d$} that are
  {$d$}-dimensional tame balls.
\newblock {\em Adv. Math.}, 410(part A):Paper No. 108716, 61, 2022.

\bibitem{DM:11}
F.~Durand and A.~Messaoudi.
\newblock Boundary of the {R}auzy fractal sets in {$\Bbb R\times\Bbb C$}
  generated by {$P(x)=x^4-x^3-x^2-x-1$}.
\newblock {\em Osaka J. Math.}, 48(2):471--496, 2011.

\bibitem{EiItoRao06}
H.~Ei, S.~Ito, and H.~Rao.
\newblock Atomic surfaces, tilings and coincidences. {II}. {R}educible case.
\newblock {\em Ann. Inst. Fourier (Grenoble)}, 56(7):2285--2313, 2006.
\newblock Num{\'e}ration, pavages, substitutions.

\bibitem{FG:15}
X.~Fu and J.-P. Gabardo.
\newblock Self-affine scaling sets in {$\Bbb R^2$}.
\newblock {\em Mem. Amer. Math. Soc.}, 233(1097):vi+85, 2015.

\bibitem{GroechenigHaas:94}
K.~Gr{\"o}chenig and A.~Haas.
\newblock Self-similar lattice tilings.
\newblock {\em J. Fourier Anal. Appl.}, 1(2):131--170, 1994.

\bibitem{GHR:99}
K.~Gr\"{o}chenig, A.~Haas, and A.~Raugi.
\newblock Self-affine tilings with several tiles. {I}.
\newblock {\em Appl. Comput. Harmon. Anal.}, 7(2):211--238, 1999.

\bibitem{GM:92}
K.~Gr\"{o}chenig and W.~R. Madych.
\newblock Multiresolution analysis, {H}aar bases, and self-similar tilings of
  {${\bf R}^n$}.
\newblock {\em IEEE Trans. Inform. Theory}, 38(2, part 2):556--568, 1992.

\bibitem{H:81}
J.~E. Hutchinson.
\newblock Fractals and self-similarity.
\newblock {\em Indiana Univ. Math. J.}, 30(5):713--747, 1981.

\bibitem{Ito-Rao:06}
S.~Ito and H.~Rao.
\newblock Atomic surfaces, tilings and coincidence. {I}. {I}rreducible case.
\newblock {\em Israel J. Math.}, 153:129--155, 2006.

\bibitem{Kenyon:92}
R.~Kenyon.
\newblock Self-replicating tilings.
\newblock In {\em Symbolic dynamics and its applications ({N}ew {H}aven, {CT},
  1991)}, volume 135 of {\em Contemp. Math.}, pages 239--263. Amer. Math. Soc.,
  Providence, RI, 1992.

\bibitem{Knuth:98}
D.~E. Knuth.
\newblock {\em The Art of Computer Programming, Vol 2: Seminumerical
  Algorithms}.
\newblock Addison Wesley, London, 3rd edition, 1998.

\bibitem{Kuratowski:68}
K.~Kuratowski.
\newblock {\em Topology. {V}ol. {II}}.
\newblock New edition, revised and augmented. Translated from the French by A.
  Kirkor. Academic Press, New York, 1968.

\bibitem{LW:96}
J.~C. Lagarias and Y.~Wang.
\newblock Integral self-affine tiles in {$\Bbb R^n$}. {I}. {S}tandard and
  nonstandard digit sets.
\newblock {\em J. London Math. Soc. (2)}, 54(1):161--179, 1996.

\bibitem{LW:96a}
J.~C. Lagarias and Y.~Wang.
\newblock Self-affine tiles in {$\Bbb R^n$}.
\newblock {\em Adv. Math.}, 121(1):21--49, 1996.

\bibitem{LW:97}
J.~C. Lagarias and Y.~Wang.
\newblock Integral self-affine tiles in {$\Bbb R^n$}. {II}. {L}attice tilings.
\newblock {\em J. Fourier Anal. Appl.}, 3(1):83--102, 1997.

\bibitem{LW:03}
J.~C. Lagarias and Y.~Wang.
\newblock Substitution {D}elone sets.
\newblock {\em Discrete Comput. Geom.}, 29(2):175--209, 2003.

\bibitem{Lai-Lau:17}
C.-K. Lai and K.-S. Lau.
\newblock Some recent developments of self-affine tiles.
\newblock In {\em Recent developments in fractals and related fields}, Trends
  Math., pages 207--232. Birkh\"{a}user/Springer, Cham, 2017.

\bibitem{Lai-Lau-Rao:17}
C.-K. Lai, K.-S. Lau, and H.~Rao.
\newblock Classification of tile digit sets as product-forms.
\newblock {\em Trans. Amer. Math. Soc.}, 369(1):623--644, 2017.

\bibitem{LQT:23}
J.-C. Liu, Q.-Q. Liu, and M.-W. Tang.
\newblock Spectral and tiling properties for a class of planar self-affine
  sets.
\newblock {\em Chaos Solitons Fractals}, 173:Paper No. 113594, 7, 2023.

\bibitem{Mauldin-Williams:88}
R.~D. Mauldin and S.~C. Williams.
\newblock Hausdorff dimension in graph directed constructions.
\newblock {\em Trans. Amer. Math. Soc.}, 309(2):811--829, 1988.

\bibitem{Rauzy:82}
G.~Rauzy.
\newblock Nombres alg\'{e}briques et substitutions.
\newblock {\em Bull. Soc. Math. France}, 110(2):147--178, 1982.

\bibitem{ST:03}
K.~Scheicher and J.~M. Thuswaldner.
\newblock Neighbours of self-affine tiles in lattice tilings.
\newblock In {\em Fractals in {G}raz 2001}, Trends Math., pages 241--262.
  Birkh\"{a}user, Basel, 2003.

\bibitem{ST:09}
A.~Siegel and J.~M. Thuswaldner.
\newblock Topological properties of {R}auzy fractals.
\newblock {\em M\'{e}m. Soc. Math. Fr. (N.S.)}, 118:140pp., 2009.

\bibitem{SirventWang02}
V.~F. Sirvent and Y.~Wang.
\newblock Self-affine tiling via substitution dynamical systems and {R}auzy
  fractals.
\newblock {\em Pacific J. Math.}, 206(2):465--485, 2002.

\bibitem{TZ:20}
J.~Thuswaldner and S.-Q. Zhang.
\newblock On self-affine tiles whose boundary is a sphere.
\newblock {\em Trans. Amer. Math. Soc.}, 373(1):491--527, 2020.

\bibitem{Thuswaldner:06}
J.~M. Thuswaldner.
\newblock Unimodular {P}isot substitutions and their associated tiles.
\newblock {\em J. Th\'eor. Nombres Bordeaux}, 18(2):487--536, 2006.

\bibitem{TZ:19}
J.~M. Thuswaldner and S.-Q. Zhang.
\newblock On self-affine tiles that are homeomorphic to a ball.
\newblock {\em Sci. China Math.}, 67(1):45--76, 2024.

\bibitem{Wang99}
Y.~Wang.
\newblock Self-affine tiles.
\newblock In {\em Advances in wavelets ({H}ong {K}ong, 1997)}, pages 261--282.
  Springer, Singapore, 1999.

\end{thebibliography}

\end{document}